\documentclass[a4paper,12 pt]{article}
\usepackage[utf8]{inputenc}
\usepackage{amsmath}
\usepackage{amssymb}
\usepackage{latexsym}
\usepackage{amsthm}
\usepackage{amsfonts}
\usepackage{graphicx}
\usepackage{color}
\usepackage{shuffle}
\usepackage{graphicx}
\usepackage{xcolor}

\newtheorem{theorem}{Theorem}
\newtheorem{lemma}{Lemma}
\newtheorem{proposition}{Proposition}
\newtheorem{corollary}{Corollary}

\def \a{\alpha}
\def \z{\zeta}

\def\l{\textbf{l}}
\def\p{\textbf{p}}
\def\N{\mathbb{N}}
\def\S{\overset{t}{\shuffle}}

\def\SH{\shuffle}

\def\SH{\shuffle}

\def\r{\rightarrow}

\def\hh{\mathfrak{h}}

\usepackage[a4paper,top=3cm,bottom=2cm,left=1.5cm,right=1.5cm,marginparwidth=1cm]{geometry}
\title{Generalized Euler decomposition formula for interpolated multiple zeta values}
\author{$\text{Pitu Sarkar}^a,~ \text{Nita Tamang}^b$ \\ $\text{Department of Mathematics, University of North Bengal}^{a,b}$\\
	West Bengal, India 734013\\ $\text{pitucob2016@gmail.com}^a$,~$\text{nita\_math@nbu.ac.in}^b$  }
\date{}

\begin{document}
	
	\maketitle
	\begin{abstract}
		In this paper, we obtain a general $t$-shuffle product formula, using which we derive a generalized Euler decomposition formula for interpolated multiple zeta values. We also provide the same formula in case of height one through two different approaches: one by combinatorial description and another one by recursive formula. 
	\end{abstract}

\section{Introduction}

Interpolated Multiple Zeta Values (IMZVs) are a class of mathematical objects that arise from the study of multiple zeta values (MZVs), which are special values of certain classes of Dirichlet series at positive integers. These values are deeply connected to numerous branches of number theory, mathematical physics, algebra (especially in the context of modular forms), quantum field theory, and the study of periods in algebraic geometry.

For an index $\l=\left(l_1,\dots,l_n\right)$ of positive integers, with $l_1 >1$ and for a variable $t$, interpolated multiple zeta value(IMZV) defined by S. Yamamoto \cite{2} is the polynomial
$$\zeta^{t}(\l)= \zeta^{t}\left( l_1,\dots, l_n\right)=\sum_{\substack{\p=\left(l_1\Box l_2\Box \dots \Box l_n\right)\\ \Box=``," \text{or} ``+"}}t^{n-\text{dep}(\p)}\zeta(\p)~(\in{\mathbb{R}[t]}),$$
where depth is given by $\text{dep}(\l)=n$, and $\zeta(\p)$ is the multiple zeta value,  defined by the convergent series
\begin{align*}
	\zeta(\l)= \zeta \left( l_1,\dots, l_n\right)=\sum_{m_1>\dots >m_n\geq 1}\dfrac{1}{m_1^{l_1}\dots m_n^{l_n}}.
\end{align*} where the summation runs over the positive integers $m_{1}, m_{2},...., m_{n}$. For $t=1$, interpolated multiple zeta value gives multiple zeta-star value which is a variant of multiple zeta value, and is defined by
\begin{align*}
	\zeta^{\star}(\l)= \zeta \left( l_1,\dots, l_n\right)=\sum_{m_1\geq \dots \geq m_n\geq 1}\dfrac{1}{m_1^{l_1}\dots m_n^{l_n}}.
\end{align*}
Therefore, interpolated multiple zeta value interpolates both multiple zeta value as well as multiple zeta-star value.
For the index $\l=\left(l_1,\dots,l_n\right)$, height is the number of $l_i$ which are greater than $1$.

Now we consider an algebraic structure of IMZVs \cite{6}.
Let $A=\{x, y\}$ be a set of two noncommutative letters. Define $A^*$ as the set of all words generated by the set $A$, which also includes the empty word $1$. Let $\hh_t =\mathbb{Q}[t]<A>$ be the noncommutative polynomial algebra over $\mathbb{Q}[t]$ generated by $A$. Define two subalgebras of $\hh_t$,~$$\hh_t^1=\mathbb{Q}[t]+\hh_ty,~ \hh_t^0= \mathbb{Q}[t]+x\hh_ty.$$
Given any $k\in {\N}$, define $z_k=x^{k-1}y$. For $t=0$, $\hh_0, \hh_0^1$, and $\hh_0^0$ are respectively denoted by $ \hh, \hh^1$, and $ \hh^0$.\\
The $t$-shuffle product $\S$ on $h_t$ is $ \mathbb{Q}[t]$-bilinear, and satisfies the rules
\begin{align*}
&{(1)} 1\S w =w \S 1 =w,\\
&{(2)} aw_1 \S bw_2 = a(w_1 \S bw_2) + b(aw_1 \S  w_2) - \delta(w_1)\rho(a)bw_2 - \delta(w_2)\rho(b)aw_1,
\end{align*}
where $ w, w_1, w_2 \in{A^*}, a,b \in{A}$, the map $\delta :A^* \rightarrow \{0,1\}$ is defined by 
$$\delta(w) =
\begin{cases}
	1& \text{if $w=1$},\\
	0& \text{if $w\neq{1}$}
\end{cases}
$$
and the map $\rho : A \rightarrow \hh_t$ is defined by $\rho(x) = 0, ~ \rho(y) = tx.$
Under the $t$-shuffle product, $\hh_t$ becomes a commutative $ \mathbb{Q}[t]$-algebra, and $\hh_t^1$, $\hh_t^0$ are also subalgebras. For $t=0$, $\overset{0}{\SH}$ is denoted by $\SH$, which is the shuffle product of multiple zeta values (see \cite{15}).\\
Define a $ \mathbb{Q}[t]$-linear map $Z^t:\hh_t^0 \rightarrow \mathbb{R}[t]$ by
$Z^t(1)=1$ and $$Z^t(x^{l_1-1}yx^{l_2-1}y \dots x^{l_n-1}y)=\z^t(l_1, \dots ,l_n),$$
where $n, l_1, l_2, \dots ,l_n$ are positive integers with $l_1 >1$. Then one can show that the map $Z^t:(\hh_t^0,\S) \rightarrow \mathbb{R}[t]$ is an algebra homomorphism i.e., we have $$\z^t(w_1 \S w_2)=\z^t(w_1)\z^t(w_2),$$ for any $w_1, w_2 \in \hh_t^0.$\\
The combinatorial description of the shuffle product \cite{1} is given by: 
\begin{align}
	&	d_1 \dots d_n \SH d_{n+1} \dots d_{n+m}=\sum_{\sigma \in {\mathbb{P}}_{n,m}} d_{{\sigma}(1)} \dots d_{{\sigma}(n+m)},
\end{align}
where $d_1, \dots ,d_{n+m}  \in A$, and  ${{\mathbb{P}}_{n,m}}$ is the set of all permutations $\sigma$ in $\{1,2, \dots , n+m\} $
with the condition
\begin{align*}
	& {\sigma}^{-1} (1) \leq {\sigma }^{-1}(2) \leq \dots \leq {\sigma}^{-1} (n);\\
	 &   {\sigma}^{-1} (n+1) \leq {\sigma}^{-1} (n+2) \leq \dots  \leq {\sigma}^{-1} (n+m).\\
\end{align*}
The shuffle product plays a pivotal role in obtaining the Euler decomposition formula for multiple zeta values \cite{11}, which is given by the expression
\begin{align}
	\label{equn1}
	\z(i)\z(j)&= \sum_{k=1}^{i} \binom{i+j-k-1}{n-1}\z(i+j-k,k) + \sum_{k=1}^{j} \binom{i+j-k-1}{i-1}\z(i+j-k,k),
\end{align}
where $i, j \geq 2.$
This formula is a fundamental result in the theory of multiple zeta values.

 In this paper, we obtain a generalized Euler decomposition formula in case of interpolated multiple zeta values (Section 2). We also derive a similar formula for height one case applying two different approaches (Section 3). The first one is by combinatorial description of shuffle product, and the second one is by recursive formula for $t$-shuffle product. In Section 4, we discuss some applications of the height one formula, and in Section 5, some more observation on a $t$-shuffle product formula is given for height 2 case. \\
 The main theorems of this paper are stated bellow: 
\begin{theorem}
	\label{T101}
	For integers ${m,u \geq 2, p \geq 1, n,v \geq 0}$, we have
	\begin{align}
		\label{Equn101}  
		&\z^t (m,\underbrace{p, \dots, p}_n) \z^t (u,\underbrace{p, \dots, p}_v)\nonumber \\
		&  = \sum_{\substack{{\a}_1+ \dots +{\a}_{n+v+2}=\\(n+v)(p-1)+(m+u)-2\\ {\a}_1, \dots , {\a}_{n+v+2} \geq 0}}  \bigg[ C_1 \times \bigg\{ \z^t({\a}_1+1, \dots, {\a}_{n+v+2}+1)-t \z^t({\a}_1+1, \dots ,{\a}_{n-l_{q+1}+v+1}+1,   \nonumber \\
		& \ \  {\a}_{n-l_{q+1}+v+2}+{\a}_{n-l_{q+1}+v+3}+2, {\a}_{n-l_{q+1}+v+4}+1, \dots , {\a}_{n+v+2}+1)    \bigg\} +C_2 \times \bigg\{ \z^t({\a}_1+1, \dots, \nonumber \\
		& \ \   {\a}_{n+v+2}+1)-t \z^t({\a}_1+1, \dots, {\a}_{n+v-n_{q}+1}+1,{\a}_{n+v-n_{q}+2}+{\a}_{n+v-n_{q}+3}+2, {\a}_{n+v-n_{q}+4}+1, \dots ,   \nonumber \\
		& \ \  {\a}_{n+v+2}+1  ) \bigg\} +C_3  \times \bigg\{ \z^t({\a}_1+1, \dots, {\a}_{n+v+2}+1)    -t \z^t({\a}_1+1, \dots, {\a}_{n-n_{q+1}+v+1}+1,               \nonumber \\
		& \ \ {\a}_{n-n_{q+1}+v+2}+ {\a}_{n-n_{q+1}+v+3}+2, {\a}_{n-n_{q+1}+v+4}+1, \dots ,{\a}_{n+v+2}+1) \bigg\}  +C_4 \times \bigg\{ \z^t({\a}_1+1, \dots,     \nonumber \\
		& \ \ {\a}_{n+v+2}+1) -t \z^t({\a}_1+1, \dots ,{\a}_{n+v-l_{q}+1}+1, {{\a}_{n+v+2-l_{q}}}+{{\a}_{n+v-l_{q}+3}}+2, {\a}_{n+v-l_{q}+4}+1, \dots ,\nonumber \\
		& \ \ {\a}_{n+v+2}+1) \bigg\} \bigg],
	\end{align}
 where
\begin{align*}
	& C_1= \sum_{\substack{l_1+ \dots + l_{q+1}=n+1\\ n_1+ \dots + n_q=v+1\\q \geq 1, l_i \geq 1, n_j \geq 1}} \prod_{i=1}^{L_q+v+1} \binom{{\a}_i}{{\beta}_i} \prod_{j=L_q+v+3}^{n+v+2} {\delta}_{{\a}_j,p-1},  \\
	& C_2= \sum_{\substack{l_1+ \dots + l_{q}=n+1\\ n_1+ \dots + n_q=v+1\\q \geq 1, l_i \geq 1, n_j \geq 1}} \prod_{i=1}^{n+N_{q-1}+1} \binom{{\a}_i}{{\beta}_i} \prod_{j=n+N_{q-1}+3}^{n+v+2} {\delta}_{{\a}_j,p-1},  \\
	& C_3= \sum_{\substack{l_1+ \dots + l_{q}=n+1\\ n_1+ \dots + n_{q+1}=v+1\\q \geq 1, l_i \geq 1, n_j \geq 1}} \prod_{i=1}^{N_q+n+1} \binom{{\a}_i}{{\gamma}_i} \prod_{j=N_q+n+3}^{n+v+2} {\delta}_{{\a}_j,p-1}, \\
\end{align*}
\begin{align*}
	& C_4= \sum_{\substack{l_1+ \dots + l_{q}=n+1\\ n_1+ \dots + n_q=v+1\\q \geq 1, l_i \geq 1, n_j \geq 1}} \prod_{i=1}^{v+L_{q-1}+1} \binom{{\a}_i}{{\gamma}_i} \prod_{j=v+L_{q-1}+3}^{n+v+2} {\delta}_{{\a}_j,p-1}. \nonumber
\end{align*}
Here for positive integers $l_1, \dots , l_q (,l_{q+1})$ and $n_1, \dots , n_q (,n_{q+1})$ are given by
\begin{align*}
	&\begin{aligned}
		{\beta}_{L_j+N_j+1} &=(L_j+N_j-1)(p-1)+(m+u)-2-\sum_{i=1}^{L_j+N_j} {\a}_i;\\
		{\beta}_{L_j+N_j+d} &=p-1, \hspace{.2em} (2 \leq d \leq l_{j+1});
	\end{aligned}
	\hspace{.4em}\text{for}~ j=0,1, \dots, q, \\
	&\begin{aligned}
		{\beta}_{L_{j+1}+N_j+1} &=(L_{j+1}+N_j-1)(p-1)+(m+u)-2-\sum_{i=1}^{L_{j+1}+N_j} {\a}_i;\\
		{\beta}_{L_{j+1}+N_j+d} &=p-1, \hspace{.1em} (2 \leq d \leq n_{j+1});
	\end{aligned}
	\hspace{.3em}\text{for}~ j=0,1, \dots, q-1,
\end{align*}
and 
\begin{align*}
	&\begin{aligned}
		{\gamma}_{L_j+N_j+1} &=(L_j+N_j-1)(p-1)+(m+u)-2-\sum_{i=1}^{L_j+N_j} {\a}_i;\\
		{\gamma}_{L_j+N_j+d} &=p-1, \hspace{.2em} (2 \leq d \leq n_{j+1})\\
	\end{aligned}
	\hspace{.4em}\text{for}~ j=0,1, \dots, q, \\
	&\begin{aligned}
		{\gamma}_{L_{j}+N_{j+1}+1} &=(N_{j+1}+L_j-1)(p-1)+(m+u)-2-\sum_{i=1}^{N_{j+1}+L_j} {\a}_i;\\
		{\gamma}_{L_{j}+N_{j+1}+d} &=p-1, \hspace{.2em} (2 \leq d \leq l_{j+1})
	\end{aligned}
	\hspace{.2em}\text{for}~ j=0,1, \dots, q-1,  \\ \nonumber
\end{align*}
with $L_j=l_1+l_2+ \dots +l_j,$~ $N_j= n_1+ \dots + n_j$ for $j \geq 0$ and $ L_0=N_0=0 $, and $\delta$ is Kronecker's delta symbol given by
$$\delta(w) =
\begin{cases}
	1& \text{if w=1},\\
	0& \text{if $w\neq{1}$}
\end{cases}
.$$
\end{theorem}
\begin{theorem}
	\label{Tm1} 
	For any positive integers $a, b, r, s$, we have
	\begin{align*}
		&{\z}^t(a+1,\underbrace{1, \dots , 1}_{r-1}) {\z}^t(b+1,\underbrace{1, \dots , 1}_{s-1}) \nonumber \\
		&= \sum_{\substack{{\a}_1+ \dots + {\a}_{r+s}=a+b;\\ {\a}_1, \dots , {\a}_{r+s} \geq 0}} \bigg\{\sum_{l=1}^{r} \binom{{\a}_1}{a} \binom{r+s-l-1}{r-l} \prod_{j=l+2}^{r+s}{ \delta}_{{\a}_j,0}  + \sum_{l=1}^{s} \binom{{\a}_1}{b} \binom{r+s-l-1}{s-l} \nonumber \\
		& \ \ \ \ \prod_{j=l+2}^{r+s}{ \delta}_{{\a}_j,0} \bigg\} {\z}^t({\a}_1+1, \dots , {\a}_{r+s}+1) -t \bigg[ \sum_{l=1}^{r-1} \sum_{\substack{{\a}_1+ \dots + {\a}_{l+1}=a+b;\\ {\a}_1, \dots , {\a}_{l+1} \geq 0}} \sum_{i=\text{min}\{r-l, s-1\}-1}^{r+s-l-3} \binom{{\a}_1}{a} \nonumber \\
		& \ \ \ \times \bigg \{\binom{i}{r-l-1} + \binom{i}{s-2} \bigg\}  +  \sum_{l=1}^{s-1} \sum_{\substack{{\a}_1+ \dots + {\a}_{l+1}=a+b;\\ {\a}_1, \dots , {\a}_{l+1} \geq 0}} \sum_{i=\text{min}\{r-1, s-l\}-1}^{r+s-l-3} \binom{{\a}_1}{b} \bigg \{\binom{i}{s-l-1} \nonumber \\
		& \ \ \ + \binom{i}{r-2} \bigg\}\bigg]  {\z}^t({\a}_1+1, \dots , {\a}_{l+1}+1, {\{1\}}^i,2, {\{1\}}^{r+s-l-i-3})    \nonumber \\
	\end{align*}
\begin{align}
\label{eq39}
		& \ \ -t \bigg[{\delta}_{s,1}\sum_{l=1}^{r-1} \sum_{\substack{{\a}_1+ \dots + {\a}_{l+1}=a+b;\\ {\a}_1, \dots , {\a}_{l+1} \geq 0}} \times \binom{{\a}_1}{a} {\z}^t({\a}_1+1, \dots , {\a}_{l}+1,{\a}_{l+1}+2, {\{1\}}^{r-l-1}) \nonumber \\ 
		  &+{\delta}_{r,1}\sum_{l=1}^{s-1} \sum_{\substack{{\a}_1+ \dots + {\a}_{l+1}=a+b;\\ {\a}_1, \dots , {\a}_{l+1} \geq 0}} \binom{{\a}_1}{b} {\z}^t({\a}_1+1, \dots , {\a}_{l}+1,{\a}_{l+1}+2, {\{1\}}^{s-l-1}) \nonumber \\
		&     +  \sum_{\substack{{\a}_1+ \dots + {\a}_{r+1}=a+b;\\ {\a}_1, \dots , {\a}_{r+1} \geq 0}} \binom{{\a}_1}{a}  {\z}^t({\a}_1+1, \dots , {\a}_{r-1}+1, {\a}_{r}+{\a}_{r+1} +2, {\{1\}}^{s-1}) \nonumber \\
		&    + \sum_{\substack{{\a}_1+ \dots + {\a}_{s+1}=a+b;\\ {\a}_1, \dots , {\a}_{s+1} \geq 0}} \binom{{\a}_1}{b} {\z}^t({\a}_1+1, \dots , {\a}_{s-1}+1,{\a}_{s}+{\a}_{s+1}+2, {\{1\}}^{r-1}) \bigg].
	\end{align}
	
\end{theorem}
	\begin{theorem}
		\label{Tm2}
	For any positive integers $m, n, j, k$, we have
	\begin{align}
		\label{eq430}
		&{\z}^t(m+1,\underbrace{1, \dots , 1}_{j-1}) {\z}^t(n+1,\underbrace{1, \dots , 1}_{k-1}) \nonumber \\
		&=\sum_{0 \leq  n_1 \leq n} \binom{m+n_1-1}{m-1} \sum_{\substack{m_1+m_2=j\\ m_i \geq 0}}\binom{m_2+k-1}{k-1}  \sum_{\substack{a_1+ \dots + a_{m_1+1}=n-n_1\\ a_i \geq 0}} \z^t(a_1+m+n_1+1, a_2+1,  \nonumber \\
		& \ \ \  \dots , a_{m_1+1}+1,\{1\}^{m_2+k-1}) + \sum_{ 1 \leq k_1 \leq k}\sum_{\substack{m_1+m_2=m-1\\ m_i \geq 0}} \binom{m_1+n-1}{n-1}\sum_{\substack{b_1+ \dots + b_{k_1+1}=m_2\\ b_i \geq 0}}\binom{j+k-k_1}{j}
		\nonumber \\
		& \ \ \ \times   \z^t(b_1+n+m_1+1, b_2+1, \dots , b_{k_1}+1,b_{m_1+1}+2, \{1\}^{j+k-k_1-1}) \nonumber \\
		& -t \sum_{0 \leq  n_1 \leq n} \binom{m+n_1-1}{m-1}  \sum_{\substack{m_1+m_2=j\\ m_i \geq 0}}\sum_{\substack{a_1+ \dots + a_{m_1+1}=n-n_1\\ a_i \geq 0}} \sum_{i=min \{m_2, k-1\}-1}^{m_2+k-3} \bigg\{ \binom{i}{m_2-1} 
		\nonumber \\
		& \ \  \ +  \binom{i}{k-2}\bigg\} \times \z^t(a_1+m+n_1+1, a_2+1, \dots , a_{m_1+1}+1, \{1\}^{i}, 2, \{1\}^{m_2+k-i-3} )  \nonumber \\
		& -t\sum_{0 \leq  n_1 \leq n} \binom{m+n_1-1}{m-1}  \sum_{\substack{j_1+j_2=n-n_1\\ j_i \geq 0}} \sum_{\substack{a_1+ \dots + a_{j}=j_1\\ a_i \geq 0}}
		\z^t(a_1+m+n_1+1, a_2+1, \dots ,  \nonumber \\
		& \ \ \  a_{j-1}+1, a_j+j_2+2, \{1\}^{k-1} ) -t \sum_{ 1 \leq k_1 \leq k}\sum_{\substack{m_1+m_2=m-1\\ m_i \geq 0}} \binom{m_1+n-1}{n-1} \sum_{i=min \{j,   k-k_1\}-1}^{j+k-k_1-2} \nonumber \\
		& \ \ \  \bigg\{ \binom{i}{j-1}+\binom{i}{k-k_1-1}\bigg\}  \sum_{\substack{a_1+ \dots + a_{k_1+1}=m_2\\ a_l \geq 0}}\z^t(a_1+n+m_1+1,  \nonumber \\
		& \ \ \  a_2+1, \dots , a_{k_1}+1, a_{k_1+1}+2, \{1\}^{i-1}, 2, \{1\}^{j+k-k_1-i-2} )
		\nonumber \\
		& -t \sum_{\substack{m_1+m_2+m_3=m-1\\ m_i \geq 0}} \binom{m_1+n-1}{n-1}\sum_{\substack{a_1+ \dots + a_{k}=m_2\\ a_l \geq 0}}
		\z^t(a_1+n+m_1+1, a_2+1, \dots, a_{k-1}+1, \nonumber \\
		& \ \ \ a_{k}+m_3+3,  \{1\}^{j-1} ).
	\end{align}
\end{theorem}
The above two theorems of the height one case are equivalent but their appearances are different. The equivalence of these two theorems is remarked in section 3.  
\section{Proof of Theorem \ref{T101}}	
We first derive a general $t$-shuffle product formula (Theorem \ref{th4}), which generalizes the formula for shuffle product \cite[Theorem 2.1]{1}. Using it, we prove Theorem \ref{T101}. For positive integer $k$, consider $\{1\}^k=\underbrace{1, \dots, 1}_{k ~\text{times}}.$ 
	\begin{theorem}
		\label{th4}
		Let $r,s$ be two positive integers and let $a_1, \dots , a_r, b_1, \dots , b_s$ be nonnegative integers. Then we have 
		\begin{align}
			\label{equ34}
			& x^{a_1}y \dots x^{a_r}y \S x^{b_1}y \dots x^{b_s}y \nonumber \\
			&  = \sum_{\substack{{\a}_1+ \dots +{\a}_{r+s}=\sum_{i=1}^{r}a_i+\sum_{j=1}^{s}b_j \\ {\a}_1, \dots , {\a}_{r+s} \geq 0}}  \bigg[ {\sum}_1 \bigg\{ x^{{\a}_1}yx^{{\a}_2}y \dots x^{{\a}_{r+s}}y\nonumber \\
			& \ \ \  -t x^{{\a}_1}y \dots x^{{\a}_{r-l_{p+1}+s-1}}y x^{{\a}_{r-l_{p+1}+s}+{\a}_{r-l_{p+1}+s+1}+1}y x^{{\a}_{r-l_{p+1}+s+2}}y \dots x^{{\a}_{r+s}}y \bigg\} \nonumber \\
			& \ \ + {\sum}_2 \bigg\{ x^{{\a}_1}y \dots x^{{\a}_{r+s}}y -t x^{{\a}_1}y \dots x^{{\a}_{r+s-n_{p}-1}}y x^{{\a}_{r+s-n_{p}}+{\a}_{r+s-n_{p}+1}+1}y x^{{\a}_{r+s-n_{p}+2}}y \dots x^{{\a}_{r+s}}y \bigg\} \nonumber \\
			&  \ \ + {\sum}_3 \bigg\{ x^{{\a}_1}y \dots x^{{\a}_{r+s}}y   -t x^{{\a}_1}y \dots x^{{\a}_{s-n_{p+1}+r-1}}y x^{{\a}_{s-n_{p+1}+r}+{\a}_{s-n_{p+1}+r+1}+1}y x^{{\a}_{s-n_{p+1}+r+2}}y \dots x^{{\a}_{r+s}}y \bigg\} \nonumber \\
			& \ \  + {\sum}_4 \bigg\{ x^{{\a}_1}y \dots x^{{\a}_{r+s}}y -t x^{{\a}_1}y \dots x^{{\a}_{r+s-l_{p}-1}}y x^{{\a}_{r+s-l_{p}}+{\a}_{r+s-l_{p}+1}+1}y x^{{\a}_{r+s-l_{p}+2}}y \dots x^{{\a}_{r+s}}y \bigg\} \bigg], 
		\end{align} 
		where 
		\begin{align}
			\label{b1}
			& {\sum}_1= \sum_{\substack{l_1+ \dots + l_{p+1}=r\\ n_1+ \dots + n_p=s\\p \geq 1, l_i \geq 1, n_j \geq 1}} \prod_{i=1}^{L_p+s} \binom{{\a}_i}{{\beta}_i} \prod_{j=L_p+s+2}^{r+s} {\delta}_{{\a}_j,a_{j-s}}, 
		\end{align}
			\begin{align}
				\label{b2}
			& {\sum}_2= \sum_{\substack{l_1+ \dots + l_{p}=r\\ n_1+ \dots + n_p=s\\p \geq 1, l_i \geq 1, n_j \geq 1}} \prod_{i=1}^{r+N_{p-1}} \binom{{\a}_i}{{\beta}_i} \prod_{j=r+N_{p-1}+2}^{r+s} {\delta}_{{\a}_j,b_{j-r}},  
				\end{align}
			\begin{align}
				\label{b3}
			& {\sum}_3= \sum_{\substack{l_1+ \dots + l_{p}=r\\ n_1+ \dots + n_{p+1}=s\\p \geq 1, l_i \geq 1, n_j \geq 1}} \prod_{i=1}^{N_p+r} \binom{{\a}_i}{{\gamma}_i} \prod_{j=N_p+r+2}^{r+s} {\delta}_{{\a}_j,b_{j-r}}, 
				\end{align}
			\begin{align}
				\label{b4}
			& {\sum}_4= \sum_{\substack{l_1+ \dots + l_{p}=r\\ n_1+ \dots + n_p=s\\p \geq 1, l_i \geq 1, n_j \geq 1}} \prod_{i=1}^{s+L_{p-1}} \binom{{\a}_i}{{\gamma}_i} \prod_{j=s+L_{p-1}+2}^{r+s} {\delta}_{{\a}_j,a_{j-s}}.  
		\end{align}
			Here for positive integers $l_1, \dots , l_p (,l_{p+1})$ and $n_1, \dots , n_p (,n_{p+1})$ appearing in the summations above, are given by
		\begin{align}
			&\begin{aligned}
				\label{b5}
				{\beta}_{L_j+N_j+1} &=\sum_{i=1}^{L_j+1} a_i+ \sum_{i=1}^{N_j} b_i-\sum_{i=1}^{L_j+N_j} {\a}_i,\\
				{\beta}_{L_j+N_j+d} &=a_{L_j+d}, \hspace{1em} (2 \leq d \leq l_{j+1})
			\end{aligned}
			\hspace{4em}\text{for}~ j=0,1, \dots, p, \\
			&\begin{aligned}
				\label{b6}
				{\beta}_{L_{j+1}+N_j+1} &=\sum_{i=1}^{L_{j+1}} a_i+ \sum_{i=1}^{N_j+1} b_i-\sum_{i=1}^{L_{j+1}+N_j} {\a}_i,\\
				{\beta}_{L_{j+1}+N_j+d} &=b_{N_j+d}, \hspace{1em} (2 \leq d \leq n_{j+1})
			\end{aligned}
			\hspace{4em}\text{for}~ j=0,1, \dots, p-1, 
		\end{align}
		and 
		\begin{align}
			&\begin{aligned}
				\label{b7}
				{\gamma}_{L_j+N_j+1} &=\sum_{i=1}^{L_j} a_i+ \sum_{i=1}^{N_j+1} b_i-\sum_{i=1}^{L_j+N_j} {\a}_i,\\
				{\gamma}_{L_j+N_j+d} &=b_{N_j+d}, \hspace{1em} (2 \leq d \leq n_{j+1})\\
			\end{aligned}
			\hspace{4em}\text{for}~ j=0,1, \dots, p, \\
			&\begin{aligned}
				\label{b8}
				{\gamma}_{L_{j}+N_{j+1}+1} &=\sum_{i=1}^{L_{j}+1} a_i+ \sum_{i=1}^{N_{j+1}} b_i-\sum_{i=1}^{N_{j+1}+L_j} {\a}_i,\\
				{\gamma}_{L_{j}+N_{j+1}+d} &=a_{L_j+d}, \hspace{1em} (2 \leq d \leq l_{j+1})
			\end{aligned}
			\hspace{4em}\text{for}~ j=0,1, \dots, p-1, \\ \nonumber
		\end{align}
		with $L_j=l_1+l_2+ \dots +l_j,$~ $N_j= n_1+ \dots + n_j$ for $j \geq 0$ and $ L_0=N_0=0 $, and $\delta$ is Kronecker's delta symbol.
		
	\end{theorem}
	
	\begin{proof}
		We follow the techniques used to obtain the shuffle product formula by Z. Li and C. Qin in \cite{1}. If we take $y$ in $ x^{a_1}y \dots x^{a_r}y $ as $y_1$ and the $y$ in $x^{b_1}y \dots x^{b_s}y$ as $y_2$, then there are four cases for the positions of $y_1$'s and $y_2$'s when we do $t$-shuffle product:
	\begin{align*}
			&(i): \underbrace{y_1 \dots y_1}_{l_1}\underbrace{y_2 \dots y_2}_{n_1} \dots \underbrace{y_1 \dots y_1}_{l_p} \underbrace{y_2 \dots y_2}_{n_p-1} (y_2-tx)\underbrace{y_1 \dots y_1}_{l_{p+1}} \\
			& \ \ \ \ 	\text{	where}~ l_1+\dots + l_{p+1}=r, n_1+\dots + n_{p}=s ~\text{	with}~ p,l_i,n_j \geq 1;\\
			& (ii): \underbrace{y_1 \dots y_1}_{l_1}\underbrace{y_2 \dots y_2}_{n_1} \dots \underbrace{y_1 \dots y_1}_{l_p-1} (y_1-tx) \underbrace{y_2 \dots y_2}_{n_p} \\
			&\ \ \ \ 	\text{	where}~ l_1+\dots + l_{p}=r, n_1+\dots + n_{p}=s ~\text{	with}~ p,l_i,n_j \geq 1;\\
			&(iii): \underbrace{y_2 \dots y_2}_{n_1}\underbrace{y_1 \dots y_1}_{l_1} \dots \underbrace{y_2 \dots y_2}_{n_p}\underbrace{y_1 \dots y_1}_{l_{p}-1} (y_1-tx)\underbrace{y_2 \dots y_2}_{n_{p+1}}  \\
			&\ \ \ \ 	\text{	where}~ l_1+\dots + l_{p}=r, n_1+\dots + n_{p+1}=s ~\text{	with}~ p,l_i,n_j \geq 1;\\
			&(iv): \underbrace{y_2 \dots y_2}_{n_1}\underbrace{y_1 \dots y_1}_{l_1} \dots \underbrace{y_2 \dots y_2}_{n_p-1} (y_2-tx)\underbrace{y_1 \dots y_1}_{l_{p}}  \\
			&\ \ \ \ 	\text{	where}~ l_1+\dots + l_{p}=r, n_1+\dots + n_{p}=s ~\text{	with}~ p,l_i,n_j \geq 1;
		\end{align*} 
Thus the $t$-shuffle product gives all the terms of the shuffle product formula from \cite{1} and additionally new terms involving t. In case $(i)$, it can be easily seen that, replacing $s^{th}$~$y_2$, by $-tx$ i.e., $(r+s-l_{p+1})^{th}$ $y$, by $-tx$ with the same coefficient as the corresponding words from the shuffle product, we get the new words. 
Therefore in this case new terms are given by
\begin{align*}
&-t\sum_{\substack{{\a}_1+ \dots +{\a}_{r+s}=\sum_{i=1}^{r}a_i+\sum_{j=1}^{s}b_j \\ {\a}_1, \dots , {\a}_{r+s} \geq 0}}  {\sum}_1   x^{{\a}_1}y \dots x^{{\a}_{r-l_{p+1}+s-1}}y x^{{\a}_{r-l_{p+1}+s}+{\a}_{r-l_{p+1}+s+1}+1}y x^{{\a}_{r-l_{p+1}+s+2}}y \dots x^{{\a}_{r+s}}y.
\end{align*}
For case $(ii)$, $t$-shuffle product will have new words by replacing $r^{th}$~$y_1$, with $-tx$ i.e., $(r+s-n_p)^{th}$ $y$, by $-tx$ with the same coefficient as the corresponding words in the shuffle product, and in this case, new terms are given by
\begin{align*}
	&-t\sum_{\substack{{\a}_1+ \dots +{\a}_{r+s}=\sum_{i=1}^{r}a_i+\sum_{j=1}^{s}b_j \\ {\a}_1, \dots , {\a}_{r+s} \geq 0}}  {\sum}_2 x^{{\a}_1}y \dots x^{{\a}_{r+s-n_{p}-1}}y x^{{\a}_{r+s-n_{p}}+{\a}_{r+s-n_{p}+1}+1}y x^{{\a}_{r+s-n_{p}+2}}y \dots x^{{\a}_{r+s}}y. 
\end{align*}
Now for case $(iii)$, $t$-shuffle product will have new words by replacing $r^{th}$~$y_1$, with $-tx$ i.e., $(r+s-n_{p+1})^{th}$ $y$, by $-tx$ with the same coefficient as shuffle product, and new terms are given by
\begin{align*}
&-t\sum_{\substack{{\a}_1+ \dots +{\a}_{r+s}=\sum_{i=1}^{r}a_i+\sum_{j=1}^{s}b_j \\ {\a}_1, \dots , {\a}_{r+s} \geq 0}}  {\sum}_3 x^{{\a}_1}y \dots x^{{\a}_{s-n_{p+1}+r-1}}y x^{{\a}_{s-n_{p+1}+r}+{\a}_{s-n_{p+1}+r+1}+1}y x^{{\a}_{s-n_{p+1}+r+2}}y \dots x^{{\a}_{r+s}}y.
\end{align*}
Similarly, for case $(iv)$, we have new words by replacing $s^{th}$~$y_2$, with $-tx$ i.e., $(r+s-l_{p})^{th}$ $y$, by $-tx$ with the same coefficient as shuffle product, and new terms are given by
\begin{align*}
	&-t\sum_{\substack{{\a}_1+ \dots +{\a}_{r+s}=\sum_{i=1}^{r}a_i+\sum_{j=1}^{s}b_j \\ {\a}_1, \dots , {\a}_{r+s} \geq 0}}  {\sum}_4x^{{\a}_1}y \dots x^{{\a}_{r+s-l_{p}-1}}y x^{{\a}_{r+s-l_{p}}+{\a}_{r+s-l_{p}+1}+1}y x^{{\a}_{r+s-l_{p}+2}}y \dots x^{{\a}_{r+s}}y.
\end{align*}	
These coefficients ${\sum}_1, {\sum}_2, {\sum}_3,$ and ${\sum}_4$ are obtained in \cite[Eq.(2.3)]{1}, and given by \eqref{b1}, \eqref{b2}, \eqref{b3}, and \eqref{b4}, respectively, with the same notations $\beta_i, \gamma_i$ given by \eqref{b5}, \eqref{b6} , \eqref{b7}, and \eqref{b8} .
	\end{proof}
Now after taking  $r=n+1, s=v+1,a_1=m-1, b_1=u-1, a_2, \dots , a_{n+1}= b_2, \dots , b_{v+1}=p-1$ in Theorem \ref{th4}, if we apply the $\mathbb{Q}[t]$-linear map $Z^t$ on both sides of the resulting equation, then we get generalized Euler decomposition formula \eqref{Equn101}, that is, Theorem $1$. In Eq. \eqref{Equn101}, if we take $t=0$, then we get
\begin{align}
	\label{equn9}  
	&\z (m,\underbrace{p, \dots, p}_n) \z (u,\underbrace{p, \dots, p}_v)\nonumber \\
	&  = \sum_{\substack{{\a}_1+ \dots +{\a}_{n+v+2}=\\(n+v)(p-1)+(m+u)-2\\ {\a}_1, \dots , {\a}_{n+v+2} \geq 0}}  \bigg[ {C_1}+{C_2}+{C_3}+{C_4}\bigg] \times   \z({\a}_1+1, \dots, {\a}_{n+v+2}+1),
\end{align}
where $C_1, C_2,C_3, $ and $C_4$ have the same meaning as used in \eqref{Equn101}.
Here if we take $n=v=0$, then following the same procedure of deducing the Euler decomposition formula from \cite[Proposition 2.7]{1}, one can get the Euler decomposition formula \eqref{equn1}, from \eqref{equn9}.
	\section{Generalized Euler decomposition formula for IMZV of height one}    	
	As is already said, we derive generalized Euler decomposition formula of height one case through two different approaches. For the first method, we need to prove Lemma \ref{l1}.

	\begin{lemma}
		\label{l1}
		For positive integers $m,n$, we have
		\begin{align}
			\label{equn31}
			&y^m \S y^n \nonumber \\
			& = \binom{m+n}{n}y^{m+n}-t \sum_{i=\text{min} \{m,n\}-1}^{m+n-2} \bigg\{ \binom{i}{i-m+1}+\binom{i}{i-n+1}\bigg\}y^ixy^{m+n-i-1}.
		\end{align}
	\end{lemma}
	
	\begin{proof}
		We prove this by the combinatorial description of shuffle product. Write $y^m \S y^n$ as $y_1y_2 \dots y_m \S {y_1}'{y_2}' \dots {y_n}'$. Let $y_m=Y$ and ${y_n}'=Y'$. Then $y^m \SH y^n$ is a linear combination of the following words:
		\begin{align*}
			&(i) ~y^iYy^{m+n-i-1}, \hspace{2em} m-1 \leq i \leq m+n-2 ,\\
			&(ii) ~ y^iY'y^{m+n-i-1}, \hspace{2em} n-1 \leq i \leq m+n-2.
		\end{align*}
		Now for a word in the case $(i)$, $(m-1)$ many $y$'s in $y^i$ 
		comes from $y^m$ and the remaining $(i-m+1)$ many $y$ comes from $y^n$.
		Hence, the coefficient corresponding to this word is $\binom{i}{i-m+1}$. Similarly, for a word in the case $(ii)$, $(n-1)$ many $y$'s in $y^i$ comes from
		 $y^n$ and the remaining $(i-n+1)$ number of $y$ comes from $y^m$.
		Hence, the coefficient corresponding to this word is $\binom{i}{i-n+1}$.	
		So, \begin{align*}
			&y^m \SH y^n \\
			&=\bigg\{\sum_{i=m-1}^{m+n-2} \binom{i}{i-m+1} + \sum_{i=n-1}^{m+n-2} \binom{i}{i-n+1}\bigg\} y^{m+n}\\
			&=\binom{m+n}{m} y^{m+n}.
		\end{align*}
		Now different from the shuffle product, the $t$-shuffle product will have new words coming from the words in the case of $(i)$ by replacing $Y$ by $-tx$ with the same coefficient and in the case of $(ii)$ by replacing $Y'$ by $-tx$ with the same coefficient.
		Therefore,
		\begin{align*}
			&y^m \S y^n \nonumber \\
			&= \binom{m+n}{n}y^{m+n}-t \bigg\{\sum_{i=m-1}^{m+n-2} \binom{i}{i-m+1} y^ixy^{m+n-i-1} + \sum_{i=n-1}^{m+n-2} \binom{i}{i-n+1} y^ixy^{m+n-i-1}\bigg\} \\
			& = \binom{m+n}{n} y^{m+n}- t \sum_{i=\text{min} \{m,n\}-1}^{m+n-2} \bigg\{ \binom{i}{m-1}+\binom{i}{n-1}\bigg\}y^ixy^{m+n-i-1}.
		\end{align*} 	
	\end{proof}
\textbf	{Remark:}Lemma \ref{l1} can be obtained directly from Theorem \ref{th4} as well.\\

Next we have the following proposition:
	\begin{proposition}
		\label{j1}
		For any positive integers $a, b, r, s$, we have
		\begin{align}
			& x^ay^r \S x^by^s \nonumber \\
			&= \sum_{\substack{{\a}_1+ \dots + {\a}_{r+s}=a+b;\\ {\a}_1, \dots , {\a}_{r+s} \geq 0}} \bigg\{\sum_{l=1}^{r} \binom{{\a}_1}{a} \binom{r+s-l-1}{r-l} \prod_{j=l+2}^{r+s}{ \delta}_{{\a}_j,0}  + \sum_{l=1}^{s} \binom{{\a}_1}{b} \binom{r+s-l-1}{s-l}\nonumber \\
			& \ \ \ \times \prod_{j=l+2}^{r+s}{ \delta}_{{\a}_j,0} \bigg\} {x}^{{\a}_1}y \dots {x}^{{\a}_{r+s}}y -t \bigg[ \sum_{l=1}^{r-1} \sum_{\substack{{\a}_1+ \dots + {\a}_{l+1}=a+b;\\ {\a}_1, \dots , {\a}_{l+1} \geq 0}} \sum_{i=\text{min}\{r-l, s-1\}-1}^{r+s-l-3} \binom{{\a}_1}{a} \bigg \{\binom{i}{r-l-1} \nonumber 
		\end{align}
			\begin{align}
				\label{e33}
			& \ \ \  + \binom{i}{s-2} \bigg\}  +  \sum_{l=1}^{s-1} \sum_{\substack{{\a}_1+ \dots + {\a}_{l+1}=a+b;\\ {\a}_1, \dots , {\a}_{l+1} \geq 0}} \sum_{i=\text{min}\{r-1, s-l\}-1}^{r+s-l-3} \binom{{\a}_1}{b} \bigg \{\binom{i}{s-l-1} + \binom{i}{r-2} \bigg\}\bigg] \nonumber \\
			& \ \  \times {x}^{{\a}_1}y \dots {x}^{{\a}_{l+1}}y^{i+1}xy^{r+s-l-i-2}-t \bigg[{\delta}_{s,1}\sum_{l=1}^{r-1} \sum_{\substack{{\a}_1+ \dots + {\a}_{l+1}=a+b;\\ {\a}_1, \dots , {\a}_{l+1} \geq 0}} \binom{{\a}_1}{a} {x}^{{\a}_1}y \dots {x}^{{\a}_{l}}y{x}^{{\a}_{l+1}+1}y^{r-l} \nonumber \\
			& \ \ \ \ +{\delta}_{r,1}\sum_{l=1}^{s-1} \sum_{\substack{{\a}_1+ \dots + {\a}_{l+1}=a+b;\\ {\a}_1, \dots , {\a}_{l+1} \geq 0}} \binom{{\a}_1}{b} {x}^{{\a}_1}y \dots {x}^{{\a}_{l}}y{x}^{{\a}_{l+1}+1}y^{s-l}  +  \sum_{\substack{{\a}_1+ \dots + {\a}_{r+1}=a+b;\\ {\a}_1, \dots , {\a}_{r+1} \geq 0}} \binom{{\a}_1}{a} {x}^{{\a}_1}y \nonumber \\
			& \ \ \ \   \dots {x}^{{\a}_{r-1}}y{x}^{{\a}_{r}+{\a}_{r+1}+1}y^{s}  + \sum_{\substack{{\a}_1+ \dots + {\a}_{s+1}=a+b;\\ {\a}_1, \dots , {\a}_{s+1} \geq 0}} \binom{{\a}_1}{b} {x}^{{\a}_1}y \dots {x}^{{\a}_{s-1}}y{x}^{{\a}_{s}+{\a}_{s+1}+1}y^{r}\bigg]
		\end{align}
	\end{proposition}
\begin{proof}
	Let us take $y$ in $x^ay^r$ as $y_1$ and $y $ in $x^by^s$ as $y_2$. Then there are two possibilities for the position of y's, when we compute $x^ay^r \S x^by^s$ :
	\begin{align*}
		&(i) \underbrace{y_1 \dots y_1}_ly_2(\underbrace{y_1 \dots y_1}_{r-l} \S \underbrace{y_2 \dots y_2}_{s-1}) , \hspace{2cm} 1 \leq l \leq r \\
		&(ii) \underbrace{y_2 \dots y_2}_ly_1(\underbrace{y_1 \dots y_1}_{r-1} \S \underbrace{y_2 \dots y_2}_{s-l}) , \hspace{2cm} 1 \leq l \leq s
	\end{align*}
Using the combinatorial description of shuffle product, it is proved in \cite[Proposition 3.1]{1} that
		\begin{align}
			\label{34e}
			& x^ay^r \SH x^by^s \nonumber \\
			&= \sum_{\substack{{\a}_1+ \dots + {\a}_{r+s}=a+b;\\ {\a}_1, \dots , {\a}_{r+s} \geq 0}} \{ c_1+c_2\} {x}^{{\a}_1}y \dots {x}^{{\a}_{r+s}}y,	
		\end{align} 
	where 
	\begin{align*}
	&c_1= \sum_{l=1}^{r} \binom{{\a}_1}{a} \binom{r+s-l-1}{r-l} \prod_{j=l+2}^{r+s}{ \delta}_{{\a}_j,0}, \nonumber \\
	& c_2= \sum_{l=1}^{s} \binom{{\a}_1}{b} \binom{r+s-l-1}{s-l} \prod_{j=l+2}^{r+s}{ \delta}_{{\a}_j,0}.
\end{align*}
Different from the shuffle product, for the $t$-shuffle product, we need to find the place for $y$, which will be replaced by $-tx$.  
For case $(i)$, we have following subcases:\\
(a) When $l=r$, we have  new terms by replacing $l^{th}$ $y_1$ by $-tx$ with the same coefficient ($c_1$) as shuffle product, i.e., we have the new terms as follows:
\begin{align}
			\label{e34}
			-t	\sum_{\substack{{\a}_1+ \dots + {\a}_{r+1}=a+b;\\ {\a}_1, \dots , {\a}_{r+1} \geq 0}} \binom{{\a}_1}{a} {x}^{{\a}_1}y \dots {x}^{{\a}_{r-1}}y{x}^{{\a}_{r}+{\a}_{r+1}+1}y^{s}.
		\end{align}
(b) When $s=1$ $(r \neq l)$, we have new words by replacing $1^{st}$ $y_2$ with $-tx$, i.e., ${(l+1)}^{th}$ $y$ by $-tx$ with the same coefficient as shuffle product, i.e., we have the following new terms	
		\begin{align}
			-t{\delta}_{s,1}\sum_{l=1}^{r-1} \sum_{\substack{{\a}_1+ \dots + {\a}_{l+1}=a+b;\\ {\a}_1, \dots , {\a}_{l+1} \geq 0}} \binom{{\a}_1}{a} {x}^{{\a}_1}y \dots {x}^{{\a}_{l}}y{x}^{{\a}_{l+1}+1}y^{r-l}.
		\end{align}
(c) For $1 \leq l \leq r-1 (s \neq 1)$, we have the following new terms
		\begin{align}
			&-t \sum_{l=1}^{r-1} \sum_{\substack{{\a}_1+ \dots + {\a}_{l+1}=a+b;\\ {\a}_1, \dots , {\a}_{l+1} \geq 0}} \sum_{i=\text{min}\{r-l, s-1\}-1}^{r+s-l-3} \binom{{\a}_1}{a} \bigg \{\binom{i}{r-l-1} + \binom{i}{s-2} \bigg\} \nonumber \\
			& \ \ \ \times {x}^{{\a}_1}y \dots {x}^{{\a}_{l+1}}y^{i+1}xy^{r+s-l-i-2},
		\end{align}	
since $y^m \S y^n$ has the following new terms different from shuffle product 	
		\begin{align*}
			- t \sum_{i=\text{min} \{m,n\}-1}^{m+n-2} \bigg\{ \binom{i}{m-1}+\binom{i}{n-1}\bigg\}y^ixy^{m+n-i-1}.
		\end{align*}
Similarly, in case (ii), we have the new terms
			\begin{align}
			-t\sum_{\substack{{\a}_1+ \dots + {\a}_{s+1}=a+b;\\ {\a}_1, \dots , {\a}_{s+1} \geq 0}} \binom{{\a}_1}{b} {x}^{{\a}_1}y \dots {x}^{{\a}_{s-1}}y{x}^{{\a}_{s}+{\a}_{s+1}+1}y^{r},
		\end{align}	
		\begin{align}
			-t{\delta}_{r,1}\sum_{l=1}^{s-1} \sum_{\substack{{\a}_1+ \dots + {\a}_{l+1}=a+b;\\ {\a}_1, \dots , {\a}_{l+1} \geq 0}} \binom{{\a}_1}{b} {x}^{{\a}_1}y \dots {x}^{{\a}_{l}}y{x}^{{\a}_{l+1}+1}y^{s-l},
		\end{align}
		and
		\begin{align}
			\label{e39}
			&-t\sum_{l=1}^{s-1} \sum_{\substack{{\a}_1+ \dots + {\a}_{l+1}=a+b;\\ {\a}_1, \dots , {\a}_{l+1} \geq 0}} \sum_{i=\text{min}\{r-1, s-l\}-1}^{r+s-l-3} \binom{{\a}_1}{b} \bigg \{\binom{i}{s-l-1} + \binom{i}{r-2} \bigg\} \nonumber \\
			& \hspace*{5cm} \times {x}^{{\a}_1}y \dots {x}^{{\a}_{l+1}}y^{i+1}xy^{r+s-l-i-2}.
		\end{align}
Hence after adding all the new terms \eqref{e34}-\eqref{e39} with shuffle product formula \eqref{34e}, we obtain the $t$-shuffle product formula \eqref{e33}. 
	\end{proof}
	Applying the $\mathbb{Q}[t]$-linear map $Z^t$ on both sides of \eqref{e33}, we get the first formula for height one case, that is,  Theorem \ref{Tm1}. 
\subsection{Recursive formula for $t$-shuffle product}
Here we obtain a recursive formula for $t$-shuffle product. For that, we first obtain the following lemma.
		\begin{lemma}
		\label{lemma4}
		Let $a=a_1 a_2 \dots  a_m$ and $b=b_1 b_2 \dots  b_n$, with $ a_i, b_j \in A$. Then for a fixed $k$ with $1 \leq k \leq m$, we have,
		\begin{align}
			\label{39}
			a \S b &= \sum_{i=0}^{n} (a_1 \dots a_{k-1} \SH b_1 \dots b_i)a_k(a_{k+1} \dots a_m \S b_{i+1} \dots b_n) \nonumber \\
			& \ \ -(a_1 \dots a_{k-1} \SH b_1 \dots b_{n-1}\rho(b_n))a_ka_{k+1} \dots a_m \nonumber \\
			& \ \ - \delta_{k,m} \sum_{i=0}^{n-1} (a_1\dots a_{m-1} \SH b_1 \dots b_i) \rho(a_m)b_{i+1} \dots b_n.
		\end{align}
	\end{lemma}
	\begin{proof}
		For $m=1, n=1$, it is obvious. Let $m,n \geq 2$.
		We prove this by induction on $m+n$. Assume that \eqref{39} holds for $u+v$ with $u+v < m+n$, i.e., for a fixed $k$ with $1 \leq k \leq u$, we have,
		\begin{align*}
			&a_1 a_2 \dots  a_u \S b_1 b_2 \dots  b_v \\
			&= \sum_{i=0}^{v} (a_1 \dots a_{k-1} \SH b_1 \dots b_i)a_k(a_{k+1} \dots a_u \S b_{i+1} \dots b_v) \nonumber \\
			& \ \ -(a_1 \dots a_{k-1} \SH b_1 \dots b_{v-1}\rho(b_v))a_ka_{k+1} \dots a_u \nonumber \\
			& \ \ - \delta_{k,u} \sum_{i=0}^{v-1} (a_1\dots a_{u-1} \SH b_1 \dots b_i) \rho(a_u)b_{i+1} \dots b_v.
		\end{align*}
		By the definition of $t$-shuffle product, we have
		\begin{align}
			\label{40}
			a \S b &= a_1(a_2 \dots a_{m} \S b_1 \dots b_n)+b_1(a_1 \dots a_{m} \S b_2 \dots b_n). 
		\end{align}
		Here in both $t$-shuffle product on R.H.S of \eqref{40}, total number of $a_i$ and $ b_j$ is less than $m+n$. Applying the inductive hypothesis in second $t$-shuffle product, we have
		\begin{align*}
			a \S b &= a_1(a_2 \dots a_{m} \S b_1 \dots b_n)+b_1 \bigg\{\sum_{i=1}^{n} ( b_2 \dots b_i)a_1(a_{2} \dots a_m \S b_{i+1} \dots b_n) \nonumber \\
			& \ \ -(  b_2 \dots b_{n-1}\rho(b_n))a_1a_{2} \dots a_m \bigg\} \\
			& = \sum_{i=0}^{n} ( b_1b_2 \dots b_i)a_1(a_{2} \dots a_m \S b_{i+1} \dots b_n)- b_1b_2 \dots b_{n-1}\rho(b_n)a_1a_{2} \dots a_m.
		\end{align*}
	Hence by mathematical induction \eqref{39} is true for $k=1$.
		Let $2 \leq k \leq m$. By inductive hypothesis, \eqref{40} becomes
		\begin{align*}
			a \S b &=a_1 \bigg\{ \sum_{i=0}^{n} (a_2 \dots a_{k-1} \SH b_1 \dots b_i)a_k(a_{k+1} \dots a_m \S b_{i+1} \dots b_n) \nonumber \\
			& \ \ -(a_2 \dots a_{k-1} \SH b_1 \dots b_{n-1}\rho(b_n))a_ka_{k+1} \dots a_m \nonumber \\
			& \ \ - \delta_{k,m} \sum_{i=0}^{n-1} (a_2\dots a_{m-1} \SH b_1 \dots b_i) \rho(a_m)b_{i+1} \dots b_n\bigg\} \\
			&+b_1\bigg\{ \sum_{i=1}^{n} (a_1 \dots a_{k-1} \SH b_2 \dots b_i)a_k(a_{k+1} \dots a_m \S b_{i+1} \dots b_n) \nonumber \\
			& \ \ -(a_1 \dots a_{k-1} \SH b_2 \dots b_{n-1}\rho(b_n))a_ka_{k+1} \dots a_m \nonumber \\
			& \ \ - \delta_{k,m} \sum_{i=1}^{n-1} (a_1\dots a_{m-1} \SH b_2 \dots b_i) \rho(a_m)b_{i+1} \dots b_n\bigg\} \\
			&=\sum_{i=0}^{n} (a_1 \dots a_{k-1} \SH b_1 \dots b_i)a_k(a_{k+1} \dots a_m \S b_{i+1} \dots b_n) \nonumber \\
			& \ \ -(a_1 \dots a_{k-1} \SH b_1 \dots b_{n-1}\rho(b_n))a_ka_{k+1} \dots a_m \nonumber \\
			& \ \ - \delta_{k,m} \sum_{i=0}^{n-1} (a_1\dots a_{m-1} \SH b_1 \dots b_i) \rho(a_m)b_{i+1} \dots b_n.
		\end{align*}
	Therefore, by mathematical induction we can complete the proof.
		\end{proof}
	By the above lemma, we obtain the following recursive formula for $t$-shuffle product.
	\begin{theorem} 
		\label{th4.2}
		Let $k,l \geq 1$. For $ a_1, \dots a_k, b_1, \dots b_l \in A$ and $ m_1, \dots m_k, n_1, \dots n_l \in \mathbb{Z}_{\geq 0}$, we have
		\begin{align}
			\label{41}
			&(a_1^{m_1} \dots a_k^{m_k}) \S (b_1^{n_1} \dots b_l^{n_l})\nonumber \\
			&= \sum_{1 \leq j \leq l} \sum_{\substack{1 \leq n_{j_1} \leq n_j \\ n_{j_1}+n_{j_2}=n_j}} \{(a_1^{m_1-1} \SH b_1^{n_1} \dots b_j^{n_{j_1}})a_1(a_2^{m_2} \dots a_k^{m_k} \S b_j^{n_{j_2}} \dots  b_l^{n_l})\}\nonumber \\
			& \ \ \ + a_1^{m_1} \{(a_2^{m_2} \dots a_k^{m_k}) \S (b_1^{n_1} \dots b_l^{n_l})\}\nonumber \\
			& \ \ \ -\delta_{k,1}\bigg[\sum_{j=1}^{l-1}\sum_{\substack{1 \leq n_{j_1} \leq n_j \\ n_{j_1}+n_{j_2}=n_j}} (a_1^{m_1-1} \SH b_1^{n_1} \dots b_j^{n_{j_1}})\rho(a_1) b_j^{n_{j_2}} \dots  b_l^{n_l}+ a_1^{m_1-1}\rho(a_1)b_1^{n_1} \dots b_l^{n_l}\bigg] \nonumber \\
			& \ \ \ - (a_1^{m_1}  \SH (b_1^{n_1} \dots b_l^{n_l-1}\rho(b_l)))a_1a_2^{m_2} \dots a_k^{m_k}.
		\end{align}
	\end{theorem}
	\begin{proof}
	Taking $k=m_1$ in Lemma \ref{lemma4}, we get the proof.	
	\end{proof}
	
	\begin{lemma}
		For integers $m, n \geq 0$, we have
		\begin{align}
			x^m \S y^n =\sum_{\substack{m_1+ \dots +m_{n+1}=m\\ m_i \geq 0}} x^{m_1}yx^{m_2}y \dots x^{m_n}yx^{m_{n+1}}-t \sum_{i=0}^{m-1}\sum_{\substack{m_1+ \dots +m_{n}=i\\ m_j \geq 0}} x^{m_1}yx^{m_2}y \dots x^{m_{n-1}}yx^{m_{n}+m-i+1}.
		\end{align}
	\end{lemma}
	\begin{proof}
		We have 
		\begin{align*}
			x^m \SH y^n =\sum_{\substack{m_1+ \dots +m_{n+1}=m\\ m_i \geq 0}} x^{m_1}yx^{m_2}y \dots x^{m_n}yx^{m_{n+1}}. 
		\end{align*}
		For $t$-shuffle product, we only need to find the terms involving t. Consider the last $y$ in $y^n$ as $Y$. Then when we shuffle $x^m$ and $y^n$, $Y$ takes $n^{th}$ position, ${(n+1)}^{th}$ position, and upto ${(m+n)}^{th}$ position. For $t$-shuffle product different from shuffle product, there are new terms by replacing $Y$ by $-tx$ when $Y$ is in $n^{th}$ position, $({n+1})^{th}$ position upto $({m+n-1})^{th}$ position.
		When $Y$ is in $i^{th}$ position, $n \leq i \leq m+n-1$, the new terms involving t is given by
			\begin{align*}
			&-\sum_{i=0}^{m-1}\sum_{\substack{m_1+ \dots +m_{n}=i\\ m_j \geq 0}} x^{m_1}yx^{m_2}y \dots x^{m_{n-1}}yx^{m_{n}}\rho(y)x^{m-i} \\
			&=-t \sum_{i=0}^{m-1}\sum_{\substack{m_1+ \dots +m_{n}=i\\ m_j \geq 0}} x^{m_1}yx^{m_2}y \dots x^{m_{n-1}}yx^{m_{n}+m-i+1}.
		\end{align*}
Thus we have, 
		\begin{align}
			\label{eq42}
			x^m \S y^n & = B_{n+1}^m-tC_n^{m-1},
		\end{align}
where we denote first sum by $B_{n+1}^m$, and second sum by $C_n^{m-1}.$
\end{proof}

\textbf{Proof of Theorem \ref{Tm2}}:
		By Theorem \ref{th4.2} and equation \eqref{eq42}, we obtain
		\begin{align}
			\label{eq44}
			&y^m \S y^nx^k \nonumber \\
			&= \sum_{\substack{m_1+m_2=m\\ m_i \geq 0}}(y^{n-1} \SH y^{m_1})y(x^k \S y^{m_2})-t(y^{n-1} \SH y^{m-1}x)yx^k \nonumber \\
			& = \sum_{\substack{m_1+m_2=m\\ m_i \geq 0}} \binom{m_1+n-1}{n-1}y^{m_1+n}(B_{m_2+1}^k-tC_{m_2}^{k-1})-t\sum_{\substack{n_1+n_2=n-1\\ n_i \geq 0}}\binom{n_1+m-2}{m-2}y^{n_1+m-1}B_{n_2+1}^1.
		\end{align}
		Similarly, we have
		\begin{align}
			\label{eq45}
			x^m \S y^nx^k &= 	x^m \SH y^nx^k \nonumber \\
			&=\sum_{\substack{m_1+m_2=m\\ m_i \geq 0}} \binom{m_2+k-1}{k-1} B_{n+1}^{m_1}x^{m_2+k},
		\end{align}
		\begin{align}
			\label{eq46}
			y^m \S x^ny^k &= \sum_{\substack{m_1+m_2=m\\ m_i \geq 0}}(x^n \SH y^{m_1}) y (y^{k-1} \S y^{m_2})- t(x^n \SH y^{m-1}x)y^k \nonumber \\
			&= \sum_{\substack{m_1+m_2=m\\ m_i \geq 0}} B_{m_1+1}^n y \bigg[\binom{m_2+k-1}{k-1}y^{m_2+k-1}-t \sum_{i=\text{min} \{m_2,k-1\}-1}^{m_2+k-3} \bigg\{ \binom{i}{m_2-1}+  \nonumber \\
			& \ \ \ \binom{i}{k-2}\bigg\}y^ixy^{m_2+k-i-2}\bigg] -t \sum_{\substack{n_1+n_2=n\\ n_i \geq 0}} B_{m}^{n_1} x^{n_2+1}y^k  \hspace{1cm}[\text{here we use equation \eqref{equn31}}]\nonumber \\
			&= \sum_{\substack{m_1+m_2=m\\ m_i \geq 0}} B_{m_1+1}^n  \bigg[\binom{m_2+k-1}{k-1}y^{m_2+k}-t \sum_{i=\text{min} \{m_2,k-1\}-1}^{m_2+k-3} \bigg\{ \binom{i}{m_2-1}+  \nonumber \\
			& \ \ \  \ \binom{i}{k-2}\bigg\}y^{i+1}xy^{m_2+k-i-2}\bigg] -t \sum_{\substack{n_1+n_2=n\\ n_i \geq 0}} B_{m}^{n_1} x^{n_2+1}y^k
		\end{align}
		and 
		\begin{align}
			\label{eq47}
			x^m \S x^ny^k&= \sum_{\substack{m_1+m_2=m\\ m_i \geq 0}}(x^{n-1} \SH x^{m_1}) x (y^{k} \S x^{m_2}) \nonumber \\
			&=\sum_{\substack{m_1+m_2=m\\ m_i \geq 0}} \binom{m_1+n-1}{m_1} x^{m_1+n}(B_{k+1}^{m_2}-tC_k^{m_2-1})
		\end{align}
		Using Theorem \ref{th4.2}, equations \eqref{eq46},  \eqref{equn31} and equation (18) from \cite{5}, we have
		\begin{align*}
			&x^my^j \S x^ny^k \nonumber \\
			&=\sum_{\substack{n_1+n_2=n\\ m_i \geq 0}}(x^{m-1} \SH x^{n_1})x(y^j \S x^{n_2}y^k)+
			\sum_{\substack{k_1+k_2=k\\ 1 \leq k_1 \leq k}}(x^{m-1} \SH x^{n}y^{k_1})x(y^j \S y^{k_2}) - t(x^{m-1} \SH x^ny^{k-1}x)xy^j \nonumber \\
				\end{align*}
		\begin{align}
		\label{eq48}
			&=\sum_{0 \leq  n_1 \leq n} \binom{m+n_1-1}{m-1}x^{m+n_1}\Bigg\{\sum_{\substack{m_1+m_2=j\\ m_i \geq 0}} B_{m_1+1}^{n-n_1}  \bigg[\binom{m_2+k-1}{k-1}y^{m_2+k} -t \sum_{i=\text{min} \{m_2, k-1\}-1}^{m_2+k-3} \nonumber \\
			& \ \ \ \times  \bigg\{ \binom{i}{m_2-1}+  \binom{i}{k-2}\bigg\}y^{i+1}xy^{m_2+k-i-2}\bigg] -t \sum_{\substack{j_1+j_2=n-n_1\\ j_i \geq 0}} B_{j}^{j_1} x^{j_2+1}y^k\Bigg\}  \nonumber \\
			& \ \ \ + \sum_{ 1 \leq k_1 \leq k} \sum_{\substack{m_1+m_2=m-1\\ m_i \geq 0}} \binom{m_1+n-1}{n-1} x^{m_1+n}  B_{k_1+1}^{m_2} x \Bigg[\binom{j+k-k_1}{j}y^{j+k-k_1} \nonumber \\
			& \ \ \ -t \sum_{i=min \{j, k-k_1\}-1}^{j+k-k_1-2} \bigg\{ \binom{i}{j-1}+\binom{i}{k-k_1-1}\bigg\} y^ixy^{j+k-k_1-i-1}\Bigg] \nonumber \\
			& \ \ \ -t \sum_{\substack{m_1+m_2+m_3=m-1\\ m_i \geq 0}} \binom{m_1+n-1}{n-1}x^{m_1+n}B_k^{m_2}x^{m_3+2}y^j \nonumber \\
			&=\sum_{0 \leq  n_1 \leq n} \binom{m+n_1-1}{m-1} \sum_{\substack{m_1+m_2=j\\ m_i \geq 0}}\binom{m_2+k-1}{k-1}  \sum_{\substack{a_1+ \dots + a_{m_1+1}=n-n_1\\ a_i \geq 0}} x^{a_1+m+n_1}yx^{a_2}y  \nonumber \\
			& \ \ \ \  \dots x^{a_{m_1+1}}y^{m_2+k}+ \sum_{ 1 \leq k_1 \leq k}\sum_{\substack{m_1+m_2=m-1\\ m_i \geq 0}} \binom{m_1+n-1}{n-1}\sum_{\substack{b_1+ \dots + b_{k_1+1}=m_2\\ b_i \geq 0}}\binom{j+k-k_1}{j} \times \nonumber \\
			& \ \ \ \ \   x^{b_1+n+m_1}yx^{b_2}y \dots x^{b_{k_1}}y x^{b_{m_1+1}+1}y^{j+k-k_1} \nonumber \\
			&-t \sum_{0 \leq  n_1 \leq n} \binom{m+n_1-1}{m-1} \sum_{\substack{m_1+m_2=j\\ m_i \geq 0}}\sum_{\substack{a_1+ \dots + a_{m_1+1}=n-n_1\\ a_i \geq 0}} \sum_{i=\text{min} \{m_2,k-1\}-1}^{m_2+k-3} \bigg\{ \binom{i}{m_2-1}+  \nonumber \\
			& \ \ \ \binom{i}{k-2}\bigg\}x^{a_1+m+n_1}yx^{a_2}y \dots x^{a_{m_1+1}}y^{i+1}xy^{m_2+k-i-2} \nonumber \\
			& -t\sum_{0 \leq  n_1 \leq n} \binom{m+n_1-1}{m-1} \sum_{\substack{j_1+j_2=n-n_1\\ j_i \geq 0}} \sum_{\substack{a_1+ \dots + a_{j}=j_1\\ a_i \geq 0}}x^{a_1+m+n_1}yx^{a_2}y \dots x^{a_{j-1}}yx^{a_j+j_2+1}y^k \nonumber \\
			& -t \sum_{ 1 \leq k_1 \leq k}\sum_{\substack{m_1+m_2=m-1\\ m_i \geq 0}} \binom{m_1+n-1}{n-1} \sum_{i=\text{min} \{j,k-k_1\}-1}^{j+k-k_1-2} \bigg\{ \binom{i}{j-1}+\binom{i}{k-k_1-1}\bigg\} \nonumber \\
			& \sum_{\substack{a_1+ \dots + a_{k_1+1}=m_2\\ a_l \geq 0}} x^{a_1+n+m_1}yx^{a_2}y \dots x^{a_{k_1}}yx^{a_{k_1+1}+1}y^ixy^{j+k-k_1-i-1} \nonumber \\
			& -t \sum_{\substack{m_1+m_2+m_3=m-1\\ m_i \geq 0}} \binom{m_1+n-1}{n-1}\sum_{\substack{a_1+ \dots + a_{k}=m_2\\ a_l \geq 0}}x^{a_1+n+m_1}yx^{a_2}y \dots x^{a_{k-1}}yx^{a_{k}+m_3+2}y^j.
		\end{align}
	Now applying the $\mathbb{Q}[t]$-linear map $Z^t$ on both sides of equation \eqref{eq48}, we obtain the second formula for height one case, that is, Theorem \ref{Tm2}.\\
\textbf	{Note}: If we take $t=0$ in \eqref{eq430}, we get the formula \cite[Theorem 1.1]{5}, which gives the Euler decomposition formula \cite[Corollary 1.2]{5}.\\
\textbf	{Remark:} In \cite{1}, for the case $t=0$, they showed that the formula \eqref{eq39} can be deduced from the formula \eqref{eq430}. Since the terms involved with $t$ has the same coefficients in both the shuffle and $t$-shuffle product,  one can easily deduce Theorem \ref{Tm1} from Theorem \ref{Tm2}. So we can say these two theorems are equivalent.
	\section{Some Applications}
Here we give some applications of $t$-shuffle product formula of interpolated multiple zeta values of height one. First, we obtain an alternating sum of generalized Euler decomposition for interpolated multiple zeta values, and then we deduce some relations among multiple zeta and zeta-star values. 	

\begin{proposition}
	For an even positive integer $k$, we have
	\begin{align}
		&	\sum_{i=0}^{k} {(-1)}^iz_pz_1^i \S z_pz_1^{k-i}\nonumber\\
		&= 2\bigg[2\sum_{d=1}^{\frac{k}{2}-1}{(-1)}^d\binom{k}{d}+{(-1)}^{\frac{k}{2}}\binom{k}{\frac{k}{2}}\bigg]\sum_{\substack{{\a}_1+{\a}_{2}=2(p-1)\\ {\a}_1, {\a}_{2} \geq 0}}\binom{\a_1}{p-1}  z_{\a_1+1}  z_{\a_{2}+1}z_1^{k}+2\bigg[\sum_{l=2}^{\frac{k}{2}+1}\biggl\{\sum_{d=l-1}^{\frac{k}{2}}{(-1)}^d\nonumber \\
		& \ \ \binom{k+1-l}{d+1-l}  +  \sum_{d=1}^{\frac{k}{2}-1}{(-1)}^d\binom{k+1-l}{d}\biggr\}-\sum_{l=\frac{k}{2}+2}^{k}\bigg]\sum_{\substack{{\a}_1+ \dots +{\a}_{l+1}=2(p-1)\\ {\a}_1, \dots , {\a}_{l+1} \geq 0}}\binom{\a_1}{p-1}  \nonumber \\
		&\ \ \times z_{\a_1+1} \dots z_{\a_{l+1}+1}z_1^{k+1-l} - 2t\bigg[\sum_{l=1}^k\sum_{\substack{{\a}_1+ \dots +{\a}_{l+1}=2(p-1)\\ {\a}_1, \dots , {\a}_{l+1} \geq 0}}\binom{\a_1}{p-1}z_{\a_1+1} \dots z_{\a_l+1}z_{\a_{l+1}+2}z_1^{k-l}  \nonumber \\
		& \ \  + \sum_{\substack{{\a}_1+ {\a}_{2}=2(p-1)\\ {\a}_1, {\a}_{2} \geq 0}}\binom{\a_1}{p-1}z_{\a_l+\a_{2}+2}z_1^{k} + \sum_{\substack{{\a}_1+ \dots +{\a}_{k+2}=2(p-1)\\ {\a}_1, \dots , {\a}_{k+2} \geq 0}}\binom{\a_1}{p-1}z_{\a_1+1} \dots z_{\a_k+1}z_{\a_{k+1}+\a_{k+2}+2} \Bigg] \nonumber \\
		& \ \ \  - 2t\Bigg[\sum_{i=1}^{\frac{k}{2}}\{(-1)\}^i\sum_{\substack{{\a}_1+ \dots +{\a}_{i+2}=2(p-1)\\ {\a}_1, \dots , {\a}_{i+2} \geq 0}}\binom{\a_1}{p-1}z_{\a_1+1} \dots z_{\a_i+1}z_{\a_{i+1}+\a_{i+2}+2}z_1^{k-i}\nonumber \\
		& \ \ + \sum_{i=1}^{\frac{k}{2}-1}\{(-1)\}^i\sum_{\substack{{\a}_1+ \dots +{\a}_{k-i+2}=2(p-1)\\ {\a}_1, \dots , {\a}_{k-i+2} \geq 0}}\binom{\a_1}{p-1}z_{\a_1+1} \dots z_{\a_{k-i}+1}z_{\a_{k-i+1}+\a_{k-i+2}+2}z_1^{i}\Bigg]\nonumber \\
		& \ \ -2t\Bigg[\Bigg\{\sum_{l=1}^{k-1}\{-1+{(-1)}^l\}-\delta_{k,2}\sum_{l=1}^{\frac{k}{2}}      \Bigg\}\sum_{\substack{{\a}_1+ \dots +{\a}_{l+1}=2(p-1)\\ {\a}_1, \dots , {\a}_{l+1} \geq 0}}\binom{\a_1}{p-1} z_{\a_1+1} \dots z_{\a_{l+1}+1}z_2z_1^{k-l-1}\Bigg]. 
	\end{align}
 and for an odd $k$ 
 \begin{align*}
 	&	\sum_{i=0}^{k} {(-1)}^iz_pz_1^i \S z_pz_1^{k-i}=0.\nonumber\\
\end{align*}
\end{proposition}
\begin{proof}
For odd $k$, we have
\begin{align*}
	& \sum_{i=0}^{k} {(-1)}^iz_pz_1^i \S z_pz_1^{k-i}\nonumber \\
	&=z_p \S z_p{z_1}^k-z_p{z_1} \S z_p{z_1}^{k-1}+ z_p{z_1}^2 \S z_p{z_1}^{k-2}- \dots +{(-1)}^{\frac{k-1}{2}}z_p{z_1}^{\frac{k-1}{2}} \S z_p{z_1}^{\frac{k+1}{2}} \\
	& \ \ \ + {(-1)}^{\frac{k+1}{2}}{z_p}^{\frac{k+1}{2}} \S {z_p}^{\frac{k-1}{2}}+ \dots -z_p{z_1}^{k-2} \S z_p{z_1}^{2}+z_p{z_1}^{k-1} \S z_pz_1-z_p{z_1}^k\S z_p\\
	&=0. 
\end{align*}
Suppose $k$ is even. Then
we have,
\begin{align*}
	&\sum_{i=0}^{k} {(-1)}^iz_pz_1^i \S z_pz_1^{k-i}\\
	&=2 \sum_{i=0}^{\frac{k}{2}-1} {(-1)}^i z_p{z_1}^i \S z_pz_1^{k-i}+ {(-1)}^\frac{k}{2} z_pz_1^\frac{k}{2} \S z_pz_1^{\frac{k}{2}}. 
\end{align*}
By Proposition \ref{j1}, we have
\begin{align}
	\label{j2}
	&\sum_{i=0}^{k} {(-1)}^iz_pz_1^i \S z_pz_1^{k-i}=\sum_{i=0}^{k} {(-1)}^iz_pz_1^i \SH z_pz_1^{k-i}-At=B-At.
\end{align}
where 
\begin{align*}
	&B=2\sum_{i=1}^{\frac{k}{2}-1}\{(-1)\}^i\Biggl\{\sum_{l=1}^{i+1}\binom{k+1-l}{i+1-l} +\sum_{l=1}^{k-i+1}\binom{k+1-l}{k-i-l+1}\Biggr\}\sum_{\substack{{\a}_1+ \dots +{\a}_{l+1}=2(p-1)\\ {\a}_1, \dots , {\a}_{l+1} \geq 0}}\binom{\a_1}{p-1}\nonumber \\
	& \ \ \times z_{\a_1+1} \dots z_{\a_{l+1}+1}z_1^{k+1-l} +2{(-1)}^{\frac{k}{2}}\sum_{l=1}^{\frac{k}{2}+1}\binom{k+1-l}{\frac{k}{2}+1-l}\sum_{\substack{{\a}_1+ \dots +{\a}_{l+1}=2(p-1)\\ {\a}_1, \dots , {\a}_{l+1} \geq 0}}\binom{\a_1}{p-1}\nonumber \\
	& \ \ \times z_{\a_1+1} \dots z_{\a_{l+1}+1}z_1^{k+1-l}\nonumber \\
	&=2\Biggl\{\sum_{i=1}^{\frac{k}{2}}{(-1)}^i\sum_{l=1}^{i+1}\binom{k+1-l}{i+1-l} +\sum_{i=1}^{\frac{k}{2}-1}\{(-1)\}^i\sum_{l=1}^{k-i+1}\binom{k+1-l}{i}\Biggr\}\sum_{\substack{{\a}_1+ \dots +{\a}_{l+1}=2(p-1)\\ {\a}_1, \dots , {\a}_{l+1} \geq 0}}\binom{\a_1}{p-1}\nonumber \\
	& \ \ \times z_{\a_1+1} \dots z_{\a_{l+1}+1}z_1^{k+1-l} \nonumber \\
	&= 2\sum_{d=1}^{\frac{k}{2}}{(-1)}^d\binom{k}{d}\sum_{\substack{{\a}_1+{\a}_{2}=2(p-1)\\ {\a}_1, {\a}_{2} \geq 0}}\binom{\a_1}{p-1}  z_{\a_1+1}  z_{\a_{2}+1}z_1^{k}
	+2\bigg[\sum_{l=2}^{\frac{k}{2}+1}\sum_{d=l-1}^{\frac{k}{2}}{(-1)}^d\binom{k+1-l}{d+1-l}\nonumber \\
	& \ \  + \biggl\{ \sum_{l=1}^{\frac{k}{2}+1}\sum_{d=1}^{\frac{k}{2}-1}{(-1)}^d\binom{k+1-l}{d}+\sum_{l=\frac{k}{2}+2}^{k}\sum_{d=1}^{k-l+1}{(-1)}^d\binom{k+1-l}{d}\Biggr\}\bigg]\sum_{\substack{{\a}_1+ \dots +{\a}_{l+1}=2(p-1)\\ {\a}_1, \dots , {\a}_{l+1} \geq 0}}\binom{\a_1}{p-1} \nonumber \\
	& \ \  \times z_{\a_1+1} \dots z_{\a_{l+1}+1}z_1^{k+1-l}\nonumber \\
\end{align*}
\begin{align*}
	&=2\bigg[\sum_{d=1}^{\frac{k}{2}}{(-1)}^d\binom{k}{d}+\sum_{d=1}^{\frac{k}{2}-1}{(-1)}^d\binom{k}{d}\bigg]\sum_{\substack{{\a}_1+{\a}_{2}=2(p-1)\\ {\a}_1, {\a}_{2} \geq 0}}\binom{\a_1}{p-1}  z_{\a_1+1}  z_{\a_{2}+1}z_1^{k}\nonumber \\
	& \ \ + 2\bigg[\sum_{l=2}^{\frac{k}{2}+1}\biggl\{\sum_{d=l-1}^{\frac{k}{2}}{(-1)}^d\binom{k+1-l}{d+1-l}  +  \sum_{d=1}^{\frac{k}{2}-1}{(-1)}^d\binom{k+1-l}{d}\biggr\}-\sum_{l=\frac{k}{2}+2}^{k}\bigg]\sum_{\substack{{\a}_1+ \dots +{\a}_{l+1}=2(p-1)\\ {\a}_1, \dots , {\a}_{l+1} \geq 0}}\binom{\a_1}{p-1} \nonumber \\
	& \ \  \times z_{\a_1+1} \dots z_{\a_{l+1}+1}z_1^{k+1-l}, \nonumber \\
	 &A=2\bigg[\sum_{l=1}^k\sum_{\substack{{\a}_1+ \dots +{\a}_{l+1}=2(p-1)\\ {\a}_1, \dots , {\a}_{l+1} \geq 0}}\binom{\a_1}{p-1}z_{\a_1+1} \dots z_{\a_l+1}z_{\a_{l+1}+2}z_1^{k-l} + \sum_{\substack{{\a}_1+ {\a}_{2}=2(p-1)\\ {\a}_1, {\a}_{2} \geq 0}}\binom{\a_1}{p-1}z_{\a_l+\a_{2}+2}z_1^{k} \nonumber \\
	 & \ \  + \sum_{\substack{{\a}_1+ \dots +{\a}_{k+2}=2(p-1)\\ {\a}_1, \dots , {\a}_{k+2} \geq 0}}\binom{\a_1}{p-1}z_{\a_1+1} \dots z_{\a_k+1}z_{\a_{k+1}+\a_{k+2}+2} \Bigg] + 2\Bigg[\sum_{i=1}^{\frac{k}{2}-1}{(-1)}^i\Biggl\{\Biggl[\sum_{l=1}^{i}\Biggl\{\sum_{j=i-l}^{k-l-1}\binom{j}{i-l}\nonumber \\
	 & \ \ +\sum_{j=k-i-1}^{k-l-1}\binom{j}{k-i-1}\Biggr\}+ \sum_{l=1}^{k-i}\Biggl\{\sum_{j=i-1}^{k-l-1}\binom{j}{i-1}+\sum_{j=k-i-l}^{k-l-1}\binom{j}{k-i-l}\Biggr\}\Biggr]\sum_{\substack{{\a}_1+ \dots +{\a}_{l+1}=2(p-1)\\ {\a}_1, \dots , {\a}_{l+1} \geq 0}}\binom{\a_1}{p-1}\nonumber \\
	 & \ \ \times z_{\a_1+1} \dots z_{\a_{l+1}+1}z_1^{j}z_2z_1^{k-l-j-1}+\sum_{\substack{{\a}_1+ \dots +{\a}_{i+2}=2(p-1)\\ {\a}_1, \dots , {\a}_{i+2} \geq 0}}\binom{\a_1}{p-1}z_{\a_1+1} \dots z_{\a_i+1}z_{\a_{i+1}+\a_{i+2}+2}z_1^{k-i}\nonumber \\
	 & \ \ + \sum_{\substack{{\a}_1+ \dots +{\a}_{k-i+2}=2(p-1)\\ {\a}_1, \dots , {\a}_{k-i+2} \geq 0}}\binom{\a_1}{p-1}z_{\a_1+1} \dots z_{\a_{k-i}+1}z_{\a_{k-i+1}+\a_{k-i+2}+2}z_1^{i}\Biggr\}+{(-1)}^{\frac{k}{2}}\sum_{l=1}^{\frac{k}{2}}\nonumber \\
	 & \ \ \times \Biggl\{\sum_{j=\frac{k}{2}-l}^{k-l-1}\binom{j}{\frac{k}{2}-l}+\sum_{j=\frac{k}{2}-1}^{k-l-1}\binom{j}{\frac{k}{2}-1}\Biggr\}\sum_{\substack{{\a}_1+ \dots +{\a}_{l+1}=2(p-1)\\ {\a}_1, \dots , {\a}_{l+1} \geq 0}}\binom{\a_1}{p-1} z_{\a_1+1} \dots z_{\a_{l+1}+1}z_1^{j}z_2z_1^{k-l-j-1}\nonumber \\
	 & \ \ +\{(-1)\}^{\frac{k}{2}}\sum_{\substack{{\a}_1+ \dots +{\a}_{\frac{k}{2}+2}=2(p-1)\\ {\a}_1, \dots , {\a}_{\frac{k}{2}+2} \geq 0}}\binom{\a_1}{p-1} z_{\a_1+1} \dots z_{\a_{\frac{k}{2}}+1}z_{\a_{\frac{k}{2}+1}+\a_{\frac{k}{2}+2}+2}z_1^{\frac{k}{2}}\Bigg]\nonumber \\
	&=U_1+U_2+2\Bigg[\sum_{i=1}^{\frac{k}{2}}\{(-1)\}^i\sum_{l=1}^{i}\Biggl\{\sum_{j=i-l}^{k-l-1}\binom{j}{i-l} +\sum_{j=k-i-1}^{k-l-1}\binom{j}{k-i-1}\Biggr\}  \nonumber \\
	& \ \ + \sum_{i=1}^{\frac{k}{2}-1}\{(-1)\}^i \sum_{l=1}^{k-i}\Biggl\{\sum_{j=i-1}^{k-l-1}\binom{j}{i-1}+\sum_{j=k-i-l}^{k-l-1}\binom{j}{k-i-l}\Biggr\}\Biggr]\sum_{\substack{{\a}_1+ \dots +{\a}_{l+1}=2(p-1)\\ {\a}_1, \dots , {\a}_{l+1} \geq 0}}\binom{\a_1}{p-1}\nonumber \\
	& \ \ \times z_{\a_1+1} \dots z_{\a_l+1}z_{\a_{l+1}+1}z_1^{j}z_2z_1^{k-l-j-1},
\end{align*}
where 
\begin{align*}
	&U_1=2\bigg[\sum_{l=1}^k\sum_{\substack{{\a}_1+ \dots +{\a}_{l+1}=2(p-1)\\ {\a}_1, \dots , {\a}_{l+1} \geq 0}}\binom{\a_1}{p-1}z_{\a_1+1} \dots z_{\a_l+1}z_{\a_{l+1}+2}z_1^{k-l} + \sum_{\substack{{\a}_1+ {\a}_{2}=2(p-1)\\ {\a}_1, {\a}_{2} \geq 0}}\binom{\a_1}{p-1}z_{\a_l+\a_{2}+2}z_1^{k} 
\end{align*}
 \begin{align*}
	& \ \  + \sum_{\substack{{\a}_1+ \dots +{\a}_{k+2}=2(p-1)\\ {\a}_1, \dots , {\a}_{k+2} \geq 0}}\binom{\a_1}{p-1}z_{\a_1+1} \dots z_{\a_k+1}z_{\a_{k+1}+\a_{k+2}+2} \Bigg], \nonumber \\
	&U_2=2\Bigg[\sum_{i=1}^{\frac{k}{2}}\{(-1)\}^i\sum_{\substack{{\a}_1+ \dots +{\a}_{i+2}=2(p-1)\\ {\a}_1, \dots , {\a}_{i+2} \geq 0}}\binom{\a_1}{p-1}z_{\a_1+1} \dots z_{\a_i+1}z_{\a_{i+1}+\a_{i+2}+2}z_1^{k-i}\nonumber \\
	& \ \ + \sum_{i=1}^{\frac{k}{2}-1}\{(-1)\}^i\sum_{\substack{{\a}_1+ \dots +{\a}_{k-i+2}=2(p-1)\\ {\a}_1, \dots , {\a}_{k-i+2} \geq 0}}\binom{\a_1}{p-1}z_{\a_1+1} \dots z_{\a_{k-i}+1}z_{\a_{k-i+1}+\a_{k-i+2}+2}z_1^{i}\Bigg].
\end{align*}
Now we have
\begin{align*}
	&A= U_1+U_2+2\Bigg[\sum_{l=1}^{\frac{k}{2}}\Biggl\{\sum_{d=l}^{\frac{k}{2}}{(-1)}^d\Biggl\{\sum_{j=d-l}^{k-l-1}\binom{j}{d-l} +\sum_{j=k-d-1}^{k-l-1}\binom{j}{k-d-1}\Biggr\}  \nonumber \\
	& \ \ + \sum_{d=l}^{\frac{k}{2}-1}{(-1)}^d \Biggl\{\sum_{j=d-1}^{k-l-1}\binom{j}{d-1}+\sum_{j=k-i-l}^{k-l-1}\binom{j}{k-d-l}\Biggr\}\Biggr\}+\sum_{l=\frac{k}{2}+1}^{k-1}\sum_{d=l}^{k-1}{(-1)}^d\Biggl\{\sum_{j=d-l}^{k-l-1}\binom{j}{d-l} \nonumber \\
	& \ \  +\sum_{j=k-d-1}^{k-l-1}\binom{j}{k-d-1}\Biggr\}\Biggr]\sum_{\substack{{\a}_1+ \dots +{\a}_{l+1}=2(p-1)\\ {\a}_1, \dots , {\a}_{l+1} \geq 0}}\binom{\a_1}{p-1}\times z_{\a_1+1} \dots z_{\a_{l+1}+1}z_1^{j}z_2z_1^{k-l-j-1} \nonumber \\
	&=U_1+U_2+2\Bigg[\sum_{l=1}^{\frac{k}{2}}\Biggl\{ {(-1)}^l \Biggl\{ \sum_{d=0}^{\frac{k}{2}-l}\sum_{j=0}^{d}{(-1)}^j\binom{d}{j}+\sum_{d=\frac{k}{2}-l+1}^{k-l-1}\sum_{j=0}^{\frac{k}{2}-l}{(-1)}^j\binom{d}{j}+\sum_{d=\frac{k}{2}-1}^{k-l-1}\sum_{j=\frac{k}{2}-1}^{d}{(-1)}^{j+1}\binom{d}{j}\Biggr\} \nonumber  \\
	& \ \ + \sum_{d=0}^{\frac{k}{2}-2} \sum_{j=0}^{d}{(-1)}^{j+1}\binom{d}{j}+\sum_{d=\frac{k}{2}-1}^{k-l-1}\sum_{j=0}^{\frac{k}{2}-2}{(-1)}^{j+1}\binom{d}{j}+\sum_{d=\frac{k}{2}-l+1}^{k-l-1}{(-1)}^{l}\sum_{j=\frac{k}{2}-l+1}^{d}{(-1)}^{j}\binom{d}{j}\Biggr\} \nonumber \\
	& \ \ +\sum_{l=\frac{k}{2}+1}^{k-1}\Biggl\{\sum_{d=0}^{k-l-1}\sum_{j=0}^{d}{(-1)}^{j+1}\binom{d}{j}+{(-1)}^{l}\sum_{d=0}^{k-l-1}\sum_{j=0}^{d}{(-1)}^{j}\binom{d}{j}\Biggr\}\Biggr]\nonumber \\
	& \ \ \times \sum_{\substack{{\a}_1+ \dots +{\a}_{l+1}=2(p-1)\\ {\a}_1, \dots , {\a}_{l+1} \geq 0}}\binom{\a_1}{p-1} z_{\a_1+1} \dots z_{\a_l+1}z_{\a_{l+1}+1}z_1^{j}z_2z_1^{k-l-j-1}\nonumber\\
	&=U_1+U_2+2\Bigg[\sum_{l=1}^{\frac{k}{2}}\Biggl\{ {(-1)}^l \Biggl\{ \sum_{j=0}^{0}{(-1)}^j\binom{0}{j}+\sum_{d=\frac{k}{2}-l+1}^{k-l-1}\sum_{j=0}^{\frac{k}{2}-l}{(-1)}^j\binom{d}{j}+\sum_{d=\frac{k}{2}-1}^{k-l-1}\sum_{j=\frac{k}{2}-1}^{d}{(-1)}^{j+1}\binom{d}{j}\Biggr\} \nonumber \\
	& \ \ +  \sum_{j=0}^{0}{(-1)}^{j+1}\binom{0}{j}+\sum_{d=\frac{k}{2}-1}^{k-l-1}\sum_{j=0}^{\frac{k}{2}-2}{(-1)}^{j+1}\binom{d}{j}+\sum_{d=\frac{k}{2}-l+1}^{k-l-1}{(-1)}^{l}\sum_{j=\frac{k}{2}-l+1}^{d}{(-1)}^{j}\binom{d}{j}\Biggr\} \nonumber \\
	& \ \ +\sum_{l=\frac{k}{2}+1}^{k-1}\Biggl\{\sum_{j=0}^{0}{(-1)}^{j+1}\binom{0}{j}+{(-1)}^{l}\sum_{j=0}^{0}{(-1)}^{j}\binom{0}{j}\Biggr\}\Biggr]\nonumber \\
\end{align*}
\begin{align*}
	& \ \ \times \sum_{\substack{{\a}_1+ \dots +{\a}_{l+1}=2(p-1)\\ {\a}_1, \dots , {\a}_{l+1} \geq 0}}\binom{\a_1}{p-1} z_{\a_1+1} \dots z_{\a_l+1}z_{\a_{l+1}+1}z_1^{j}z_2z_1^{k-l-j-1}\nonumber\\
	&=U_1+U_2+2\Bigg[\Biggl(\sum_{l=1}^{\frac{k}{2}}+\sum_{l=\frac{k}{2}+1}^{k-1}\Biggr)\{-1+{(-1)}^l\}\sum_{\substack{{\a}_1+ \dots +{\a}_{l+1}=2(p-1)\\ {\a}_1, \dots , {\a}_{l+1} \geq 0}}\binom{\a_1}{p-1} z_{\a_1+1} \dots z_{\a_{l+1}+1}z_2z_1^{k-l-1}\Bigg]\nonumber \\
	& \ \ + 2\sum_{l=1}^{\frac{k}{2}}\Bigg[{(-1)}^{l}\sum_{d=\frac{k}{2}-l+1}^{k-l-1}\sum_{j=0}^{d}{(-1)}^{j}\binom{d}{j}+\sum_{d=\frac{k}{2}-1}^{k-l-1}\sum_{j=0}^{d}{(-1)}^{j+1}\binom{d}{j}\Bigg]\sum_{\substack{{\a}_1+ \dots +{\a}_{l+1}=2(p-1)\\ {\a}_1, \dots , {\a}_{l+1} \geq 0}}\binom{\a_1}{p-1} \nonumber \\
	& \ \ \times  z_{\a_1+1} \dots z_{\a_{l+1}+1}z_1^{j}z_2z_1^{k-l-j-1}\nonumber\\
	& = U_1+U_2+2\Bigg[\sum_{l=1}^{k-1}\{-1+{(-1)}^l\}\sum_{\substack{{\a}_1+ \dots +{\a}_{l+1}=2(p-1)\\ {\a}_1, \dots , {\a}_{l+1} \geq 0}}\binom{\a_1}{p-1} z_{\a_1+1} \dots z_{\a_{l+1}+1}z_2z_1^{k-l-1}\Bigg]\nonumber \\
	& \ \ +2\sum_{l=1}^{\frac{k}{2}}\Bigg[\sum_{j=0}^{\frac{k}{2}-1}{(-1)}^{j+1}\binom{\frac{k}{2}-1}{j}\Bigg]\sum_{\substack{{\a}_1+ \dots +{\a}_{l+1}=2(p-1)\\ {\a}_1, \dots , {\a}_{l+1} \geq 0}}\binom{\a_1}{p-1}z_{\a_1+1} \dots z_{\a_{l+1}+1}z_1^{j}z_2z_1^{k-l-j-1}\nonumber\\
	&=U_1+U_2+2\Bigg[\Bigg\{\sum_{l=1}^{k-1}\{-1+{(-1)}^l\}-\delta_{k,2}\sum_{l=1}^{\frac{k}{2}}\Bigg\}\sum_{\substack{{\a}_1+ \dots +{\a}_{l+1}=2(p-1)\\ {\a}_1, \dots , {\a}_{l+1} \geq 0}}\binom{\a_1}{p-1} z_{\a_1+1} \dots z_{\a_{l+1}+1}z_2z_1^{k-l-1}\Bigg].
\end{align*}
Putting the values of $A, B$ in \eqref{j2} we get the result.	
 	\end{proof}
 In the case $p=2$, we get the following result.
\begin{corollary}
	For any positive integer $k$, we have
	\begin{align}
		\label{j3}
		&	\sum_{j=0}^{k} {(-1)}^jz_2z_1^j \S z_2z_1^{k-j}\nonumber\\ &  = 
		\begin{cases}
			0& \text{if $k$ is odd}\\
			2\bigg(\sum_{\substack{{a}_1+ \dots +{a}_{k+2}=1\\ {a}_1, \dots , {a}_{k+2} \geq 0}} (a_{k+2}+1)z_{a_{k+2}+2}z_{a_1+1} \dots z_{a_{k+1}+1}\\
			 +t\sum_{i=0}^{k-1}\{2{(-1)}^i-1\}z_2z_1^iz_3z_1^{k-i-1}-3tz_4z_1^k\bigg) & \text{if $k$ is even}. 
		\end{cases}
	\end{align}
\end{corollary}
\begin{proof}
	Putting $p=2$, in \eqref{j2}, we have
	\begin{align*}
		\sum_{j=0}^{k} {(-1)}^jz_2z_1^j \S z_2z_1^{k-j}=B-At.
	\end{align*}
	From \cite[Theorem2.2.]{13}, we have for $p=2$,
	\begin{align*}
		& B=\sum_{j=0}^{k} {(-1)}^jz_2z_1^j \SH z_2z_1^{k-j}=2\sum_{\substack{{a}_1+ \dots +{a}_{k+2}=1\\ {a}_1, \dots , {a}_{k+2} \geq 0}} (a_{k+2}+1)z_{a_{k+2}+2}z_{a_1+1} \dots z_{a_{k+1}+1}.
	\end{align*}
	Now for $k=2$, we have
\begin{align*}
A=6z_4z_1^2+6z_2z_1z_3-2z_2z_3z_1.
\end{align*}
For $ k\geq 4$, we have
\begin{align*}
	&A=2\bigg[\sum_{l=1}^k\sum_{\substack{{\a}_1+ \dots +{\a}_{l+1}=2\\ {\a}_1, \dots , {\a}_{l+1} \geq 0}}\a_1 z_{\a_1+1} \dots z_{\a_l+1}z_{\a_{l+1}+2}z_1^{k-l} + \sum_{\substack{{\a}_1+ {\a}_{2}=2\\ {\a}_1, {\a}_{2} \geq 0}}\a_1z_{\a_l+\a_{2}+2}z_1^{k} \nonumber \\
	& \ \  + \sum_{\substack{{\a}_1+ \dots +{\a}_{k+2}=2\\ {\a}_1, \dots , {\a}_{k+2} \geq 0}}\a_1z_{\a_1+1} \dots z_{\a_k+1}z_{\a_{k+1}+\a_{k+2}+2} + \sum_{i=1}^{\frac{k}{2}}\{(-1)\}^i\sum_{\substack{{\a}_1+ \dots +{\a}_{i+2}=2\\ {\a}_1, \dots , {\a}_{i+2} \geq 0}}\a_1 z_{\a_1+1} \dots \nonumber \\
	& \ \  \dots z_{\a_i+1}z_{\a_{i+1}+\a_{i+2}+2}z_1^{k-i}+ \sum_{i=1}^{\frac{k}{2}-1}\{(-1)\}^i\sum_{\substack{{\a}_1+ \dots +{\a}_{k-i+2}=2\\ {\a}_1, \dots , {\a}_{k-i+2} \geq 0}}\a_1z_{\a_1+1} \dots z_{\a_{k-i}+1}z_{\a_{k-i+1}+\a_{k-i+2}+2}z_1^{i} \nonumber \\
	& + \sum_{l=1}^{k-1}\{-1+{(-1)}^l\}\sum_{\substack{{\a}_1+ \dots +{\a}_{l+1}=2\\ {\a}_1, \dots , {\a}_{l+1} \geq 0}}\a_1 z_{\a_1+1} \dots z_{\a_{l+1}+1}z_2z_1^{k-l-1}\Bigg]\nonumber \\
	&=\big(2\sum_{l=1}^{k}+4\sum_{l=1}^{\frac{k}{2}}{(-1)}^l\big)z_2z_1^{l-1}z_3z_1^{k-l}+ 4\sum_{l=1}^{\frac{k}{2}-1}{(-1)}^lz_2z_1^{k-l-1}z_3z_1^{l}+6z_4z_1^k+4z_2z_1^{k-1}z_3\nonumber \\
	& \ \ +2\Bigg\{\sum_{l=1}^{k}+\sum_{l=1}^{\frac{k}{2}}{(-1)}^l\Bigg\}(2z_3z_1^{l-1}z_2z_1^{k-l}+z_2^2z_1^{l-2}z_2z_1^{k-l}+z_2z_1z_2z_1^{l-3}z_2z_1^{k-l}+ \dots +z_2z_1^{l-2}z_2^2z_1^{k-l})\nonumber 
\end{align*} 
\begin{align*}
	&\ \ +2(2z_3z_1^{k-1}z_2+ z_2^2z_1^{k-2}z_2+z_2z_1z_2z_1^{k-3}z_2+ \dots +z_2z_1^{k-2}z_2^2)\nonumber \\
	& \ \ -4\sum_{\substack{l=1,\\ l~ \text{is odd}}}^{k-1}(2z_3z_1^{l}z_2z_1^{k-l-1}+z_2^2z_1^{l-1}z_2z_1^{k-l-1}+z_2z_1z_2z_1^{l-2}z_2z_1^{k-l-1}+ \dots +z_2z_1^{l-1}z_2^2z_1^{k-l-1})\nonumber \\
	& \ \ + 2\sum_{l=1}^{\frac{k}{2}-1}{(-1)}^l (2z_3z_1^{k-l-1}z_2z_1^{l}+z_2^2z_1^{k-l-2}z_2z_1^{l}+z_2z_1z_2z_1^{k-l-3}z_2z_1^{l}+ \dots +z_2z_1^{k-l-2}z_2^2z_1^{l})\nonumber \\
	&= \Biggl\{6\sum_{\substack{l=1,\\ l~ \text{is even}}}^{\frac{k}{2}}-2\sum_{\substack{l=1,\\ l~ \text{is odd}}}^{\frac{k}{2}}+2\sum_{l=\frac{k}{2}+1}^{k}\Biggr\}z_2z_1^{l-1}z_3z_1^{k-l}+4\sum_{l=1}^{\frac{k}{2}-1}{(-1)}^lz_2z_1^{k-l-1}z_3z_1^{l}+6z_4z_1^k+4z_2z_1^{k-1}z_3\nonumber \\
	& \ \ +\Bigg\{-2\sum_{l=1}^{k}{(-1)}^l+2\sum_{l=1}^{\frac{k}{2}}{(-1)}^l\Bigg\}(2z_3z_1^{l-1}z_2z_1^{k-l}+z_2^2z_1^{l-2}z_2z_1^{k-l}+z_2z_1z_2z_1^{l-3}z_2z_1^{k-l}\nonumber \\
	&\ \ + \dots +z_2z_1^{l-2}z_2^2z_1^{k-l})+2(2z_3z_1^{k-1}z_2+ z_2^2z_1^{k-2}z_2+z_2z_1z_2z_1^{k-3}z_2+ \dots +z_2z_1^{k-2}z_2^2)\nonumber \\
	& \ \ + 2\sum_{l=1}^{\frac{k}{2}-1}{(-1)}^l (2z_3z_1^{k-l-1}z_2z_1^{l}+z_2^2z_1^{k-l-2}z_2z_1^{l}+z_2z_1z_2z_1^{k-l-3}z_2z_1^{l}+ \dots +z_2z_1^{k-l-2}z_2^2z_1^{l})\nonumber \\
\end{align*}
\begin{align*}
	&=6z_4z_1^k+ \Biggl\{6\sum_{\substack{l=1,\\ l~ \text{is even}}}^{k}-2\sum_{\substack{l=1,\\ l~ \text{is odd}}}^{k} \Biggr\}z_2z_1^{l-1}z_3z_1^{k-l} -2\sum_{l=\frac{k}{2}+1}^{k-1}(2z_3z_1^{l-1}z_2z_1^{k-l}+z_2^2z_1^{l-2}z_2z_1^{k-l}\\
	& \ \ +z_2z_1z_2z_1^{l-3}z_2z_1^{k-l}+ \dots  +z_2z_1^{l-2}z_2^2z_1^{k-l}) + 2\sum_{l=1}^{\frac{k}{2}-1}{(-1)}^l (2z_3z_1^{k-l-1}z_2z_1^{l}+z_2^2z_1^{k-l-2}z_2z_1^{l} \nonumber \\
	& \ \ +z_2z_1z_2z_1^{k-l-3}z_2z_1^{l}  + \dots +z_2z_1^{k-l-2}z_2^2z_1^{l})\nonumber \\
	&=6z_4z_1^k+\Biggl\{6\sum_{\substack{l=1,\\ l~ \text{is even}}}^{k}-2\sum_{\substack{l=1,\\ l~ \text{is odd}}}^{k} \Biggr\}z_2z_1^{l-1}z_3z_1^{k-l}. 
\end{align*}
\end{proof}	
Applying the $\mathbb{Q}[t]$-linear map $Z^t$ on both sides of \eqref{j3} we get the following result:
\begin{proposition}
		For any positive integer $k$, we have
	\begin{align}
		\label{j4}
		&	\sum_{j=0}^{k} {(-1)}^j\z^t(2,\underbrace{1, \dots , 1}_{j})  \z^t(2,\underbrace{1, \dots , 1}_{k-j})\nonumber\\ &  = 
		\begin{cases}
			0& \text{if $k$ is odd}\\
			2\bigg(\sum_{\substack{{a}_1+ \dots +{a}_{k+2}=1\\ {a}_1, \dots , {a}_{k+2} \geq 0}} (a_{k+2}+1)\z^t(a_{k+2}+2,a_1+1, \dots , a_{k+1}+1)\\
			+t\sum_{i=0}^{k-1}\{2{(-1)}^i-1\}\z^t(2,\underbrace{1, \dots , 1}_{i}, 3, \underbrace{1, \dots , 1}_{k-i-1})-3t\z^t(4, \underbrace{1, \dots , 1}_{k})\bigg) & \text{if $k$ is even}. 
		\end{cases}
	\end{align}
\end{proposition}
\textbf{Remark:} For $t=0$, \eqref{j4}, gives
\begin{align}
	&	\sum_{j=0}^{k} {(-1)}^j\z(2,\underbrace{1, \dots , 1}_{j})  \z(2,\underbrace{1, \dots , 1}_{k-j})\nonumber\\ &  = 
	\begin{cases}
		0& \text{if $k$ is odd}\\
		2\sum_{\substack{{a}_1+ \dots +{a}_{k+2}=1\\ {a}_1, \dots , {a}_{k+2} \geq 0}} (a_{k+2}+1)\z(a_{k+2}+2,a_1+1, \dots , a_{k+1}+1)& \text{if $k$ is even}. 
	\end{cases}
\end{align}
By the duality of multiple zeta values \cite{14}, and using the expression for multiple zeta-star values \cite[Eq. (5)]{13}, we get
\begin{align}
	&	\sum_{j=0}^{k} {(-1)}^j\z(2,\underbrace{1, \dots , 1}_{j})  \z(k-j+2)\nonumber  = 
	\begin{cases}
		0& \text{if $k$ is odd}\\
		2\z^{\star}(k+3,1)& \text{if $k$ is even}, 
	\end{cases}
\end{align}
which is a particular case of \cite[Theorem 2.3]{13} and is equivalent to
\begin{align}
	\label{equn35}
	&	\sum_{j=0}^{k} {(-1)}^j\z(j+2)  \z(k-j+2)  = 
	\begin{cases}
		0& \text{if $k$ is odd}\\
		2\z^{\star}(k+3,1)& \text{if $k$ is even}. 
	\end{cases}
\end{align}
Taking $k=2, 4, 6$ and so on, we get the following relations between mzvs and mzsvs:
\begin{align*}
	&\z^{\star}(5,1)=\z(2)\z(4)-\dfrac{1}{2}\z(3)\z(3)\\
	&\z^{\star}(7,1)=\z(2)\z(6)-\z(3)\z(5)+\dfrac{1}{2}\z(4)\z(4)\\
	&\z^{\star}(9,1)=\z(2)\z(8)-\z(3)\z(7)+\z(4)\z(6)-\dfrac{1}{2}\z(5)\z(5),
\end{align*}
and so on.
For $t=1$, \eqref{j4}, gives
\begin{align}
	&	\sum_{j=0}^{k} {(-1)}^j\z^{\star}(2,\underbrace{1, \dots , 1}_{j})  \z^{\star}(2,\underbrace{1, \dots , 1}_{k-j})\nonumber\\ &  = 
	\begin{cases}
		0& \text{if $k$ is odd}\\
		2\bigg(\sum_{\substack{{a}_1+ \dots +{a}_{k+2}=1\\ {a}_1, \dots , {a}_{k+2} \geq 0}} (a_{k+2}+1)\z^{\star}(a_{k+2}+2,a_1+1, \dots , a_{k+1}+1)\\
		+ \sum_{i=0}^{k-1}\{2{(-1)}^i-1\}\z^{\star}(2,\underbrace{1, \dots , 1}_{i}, 3, \underbrace{1, \dots , 1}_{k-i-1})-3t\z^{\star}(4, \underbrace{1, \dots , 1}_{k})\bigg)&\text{if $k$ is even}. 
	\end{cases}
\end{align}
\textbf{Remark:}
We have from \cite[3.7 ]{15}, for any odd positive integer $k$,
\begin{align}\label{42}
	&\z^{\star}(k-1,1)=\dfrac{k+1}{2}\z(k)-\sum_{\substack{2 \leq j \leq k-2,\\j: odd}}\z(j)  \z(k-j)
\end{align}
Also, \eqref{equn35} gives for an odd $k$,
\begin{align}
&\sum_{\substack{0 \leq j \leq k}}{(-1)}^j\z(j+2)  \z(k-j+2)=0\nonumber \\
&\implies \sum_{\substack{2 \leq j \leq k+2}}{(-1)}^j\z(j)  \z(k-j+4)=0\nonumber \\
&\implies \sum_{\substack{2 \leq j \leq k-2}}{(-1)}^j\z(j)  \z(k-j)=0\nonumber \\
&\implies \sum_{\substack{2 \leq j \leq k-2\\j:even}}\z(j)  \z(k-j)=\sum_{\substack{2 \leq j \leq k-2\\j:odd}}\z(j)  \z(k-j)..
\end{align}
Thus (\ref{42}) can also be written as
\begin{align}
	&\z^{\star}(k-1,1)=\dfrac{k+1}{2}\z(k)-\sum_{\substack{2 \leq j \leq k-2,\\j: even}}\z(j)  \z(k-j).
\end{align}
\section{Some observations}
Here we compute the $t$-shuffle product formula for $x^ay^r \S x^{b_1}y^{s_1}x^{b_2}y^{s_2} $, where $a, b_1, b_2 \geq 0, r,s_1, s_2 \geq 1$.
Consider $y$ in $x^ay^r$ as $y_1$ and $y$ in $x^{b_1}y^{s_1}x^{b_2}y^{s_2} $ as $y_2$. Then there are four cases to be considered:
\begin{align*}
	&(i) \underbrace{y_1 \dots y_1}_{r_1}y_2(\underbrace{y_1 \dots y_1}_{r_2} \SH \underbrace{y_2 \dots y_2}_{s_1-2})y_2\underbrace{y_1 \dots y_1}_{r_3}y_2(\underbrace{y_1 \dots y_1}_{r_4} \S \underbrace{y_2 \dots y_2}_{s_2-1})\\
	& \text{where}~ r_1+r_2+r_3+r_4=r~ \text{with}~ r_1 \geq 1 ~\text{and}~ r_2, r_3, r_4 \geq 0;\\
	&(ii) \underbrace{y_2 \dots y_2}_{l}y_1(\underbrace{y_1 \dots y_1}_{r_1-1} \SH \underbrace{y_2 \dots y_2}_{s_1-l-1})y_2\underbrace{y_1 \dots y_1}_{r_2}y_2(\underbrace{y_1 \dots y_1}_{r_3} \S \underbrace{y_2 \dots y_2}_{s_2-1})\\
	& \text{where}~ r_1+r_2+r_3=r~ \text{with}~ r_1 \geq 1, r_2, r_3,  \geq 0 ~\text{and}~ 1 \leq l \leq s_1-1;\\
	&(iii) \underbrace{y_2 \dots y_2}_{s_1}\underbrace{y_1 \dots y_1}_{r_1}y_2(\underbrace{y_1 \dots y_1}_{r_2} \S \underbrace{y_2 \dots y_2}_{s_2-1})\\
	& \text{where}~ r_1+r_2=r~ \text{with}~ r_1 \geq 1  ~\text{and}~ r_2  \geq 0;\\
	&(iv) \underbrace{y_2 \dots y_2}_{s_1}\underbrace{y_2 \dots y_2}_{l}y_1(\underbrace{y_1 \dots y_1}_{r-1} \S \underbrace{y_2 \dots y_2}_{s_2-l})\\
	& \text{where}~ r_1+r_2=r~ \text{with}~ r_1 \geq 1  ~\text{and}~ r_2  \geq 0.
\end{align*}
Now $x^ay^r \S x^{b_1}y^{s_1}x^{b_2}y^{s_2}=x^ay^r \SH x^{b_1}y^{s_1}x^{b_2}y^{s_2}+\mathcal{P} $, where $\mathcal{P}$ is the sum of all the terms involving $t$ and is obtained by replacing one $y$ by $-tx$ in the above four cases.
We have \cite[Eq. (3.3)]{1},
\begin{align}
	&x^ay^r \SH x^{b_1}y^{s_1}x^{b_2}y^{s_2}\nonumber \\
	&=\sum_{\substack{{\a}_1+ \dots +{\a}_{r+s_1+s_2}=a+b_1+b_2\\ {\a}_1, \dots , {\a}_{r+s_1+s_2} \geq 0}}\Biggl\{\sum_{\substack{r_1+r_2+r_3+r_4=r;\\r_1 \geq 1, r_2, r_3, r_4 \geq 0}} \binom{{\a}_1}{a}\binom{r_2+s_1-2}{r_2}\binom{r_4+s_2-1}{r_4}{\delta}_{{\a}_1+ \dots + \a_{r_1+1}, a+b_1} \nonumber \\
	& \ \ \times \prod_{k=r_1+2}^{r_1+r_2+s_1}{\delta}_{\a_k,0}\prod_{k=r_1+r_2+r_3+s_1+2}^{r+s_1+s_2}{\delta}_{\a_k,0} \Biggr\}  x^{\a_1}y \dots x^{{\a}_{r_1+r_2+r_3+s_1+1}}y^{i+1}xy^{r_4+s_2-i} \nonumber \\
	& \ \ \ + \sum_{\substack{r_1+r_2+r_3=r;\\r_1 \geq 1, r_2, r_3,  \geq 0\\ 1 \leq l \leq s_1-1}} \binom{{\a}_1}{b_1}\binom{r_1+s_1-l-2}{r_1-1}\binom{r_3+s_2-1}{r_3}{\delta}_{{\a}_1+ \dots + \a_{l+1}, a+b_1}  \prod_{k=l+2}^{r_1+s_1}{\delta}_{\a_k,0}\prod_{k=r_1+r_2+s_1+2}^{r+s_1+s_2}{\delta}_{\a_k,0}\nonumber \\
	& \ \ + \sum_{\substack{r_1+r_2=r;\\r_1 \geq 1, r_2,\geq 0}}\binom{{\a}_1}{b_1}\binom{\a_{s_1+1}}{a+b_1-\sum_{i=1}^{s_1}\a_i} \binom{r_2+s_2-1}{r_2}\prod_{k=r_1+s_1+2}^{r+s_1+s_2}{\delta}_{\a_k,0}+\sum_{l=1}^{s_2}\binom{{\a}_1}{b_1}\binom{\a_{s_1+1}}{b_2}\nonumber\\
	& \ \ \times \binom{r+s_2-l-1}{r-1}\prod_{k=s_1+l+2}^{r+s_1+s_2}{\delta}_{\a_k,0} \Biggr\}  x^{\a_1}y \dots x^{{\a}_{r+s_1+s_2}}y.
\end{align}
For case $(i)$, we have the following subcases, where we obtain new terms than in the shuffle product:\\
$(a).$ For $r_4 \neq 0$ and $s_2 \neq 1$, new term involving $t$ is,
\begin{align}
	\label{p1234}
	&-t(1-\delta_{s_2,1})\sum_{\substack{r_1+r_2+r_3+r_4=r;\\r_1, r_4 \geq 1, r_2, r_3, \geq 0; \\{\a}_1+ \dots +{\a}_{r_1+r_2+r_3+s_1+1}=a+b_1+b_2\\ {\a}_1, \dots , {\a}_{r_1+r_2+r_3+s_1+1} \geq 0}} \Biggl\{\binom{{\a}_1}{a}\binom{r_2+s_1-2}{r_2}\sum_{i=\text{min}\{r_4, s_2-1\}-1}^{r_4+s_2-3}\bigg\{\binom{i}{r_4-1}+\binom{i}{s_2-2}\bigg\} \nonumber\\
	& \ \ \hspace{2cm} {\delta}_{{\a}_1+ \dots + \a_{r_1+1}, a+b_1}\prod_{k=r_1+2}^{r_1+r_2+s_1}{\delta}_{\a_k,0} \Biggr\}x^{\a_1}y \dots x^{{\a}_{r_1+r_2+r_3+s_1+1}}y^{i+1}xy^{r_4+s_2-i-2}.
\end{align}
$(b).$ When $r_4=0, r_3 \neq 0$, we have new terms replacing $r^{th}$ $y_1$ with $-tx$, i.e., $(r_1+r_2+r_3+s_1)^{th}$ $y$ by $-tx$ with the same coefficient as shuffle product. In this case, new terms are given 
\begin{align}
	&-t\sum_{\substack{r_1+r_2+r_3=r;\\r_1, r_3 \geq 1, r_2 \geq 0; \\{\a}_1+ \dots +{\a}_{r_1+r_2+r_3+s_1+1}=a+b_1+b_2\\ {\a}_1, \dots , {\a}_{r_1+r_2+r_3+s_1+1} \geq 0}} \Biggl\{\binom{{\a}_1}{a}\binom{r_2+s_1-2}{r_2}{\delta}_{{\a}_1+ \dots + \a_{r_1+1}, a+b_1}\prod_{k=r_1+2}^{r_1+r_2+s_1}{\delta}_{\a_k,0} \Biggr\} \nonumber\\
	& \ \ \hspace{2cm} \times x^{\a_1}y \dots x^{{\a}_{r_1+r_2+r_3+s_1-1}}y   x^{{\a}_{r_1+r_2+r_3+s_1}+{\a}_{r_1+r_2+r_3+s_1+1}+1}y^{s_2}.
\end{align} 
$(c).$ When $s_2=1, r_4 \neq 0$, we have new terms replacing $(s_1+1)^{th}$ $y_2$ with $-tx$, i.e.,   $(r_1+r_2+r_3+s_1+1)^{th}$ $y$ by $-tx$ with the same coefficient as shuffle product. New terms are given by
\begin{align}
	&-t\delta_{s_2,1}\sum_{\substack{r_1+r_2+r_3+r_4=r;\\r_1, r_4 \geq 1, r_2, r_3\geq 0; \\{\a}_1+ \dots +{\a}_{r_1+r_2+r_3+s_1+1}=a+b_1+b_2\\ {\a}_1, \dots , {\a}_{r_1+r_2+r_3+s_1+1} \geq 0}} \Biggl\{\binom{{\a}_1}{a}\binom{r_2+s_1-2}{r_2}{\delta}_{{\a}_1+ \dots + \a_{r_1+1}, a+b_1}\prod_{k=r_1+2}^{r_1+r_2+s_1}{\delta}_{\a_k,0} \Biggr\} \nonumber\\
	& \ \ \hspace{2cm} \times x^{\a_1}y \dots x^{{\a}_{r_1+r_2+r_3+s_1}}y   x^{{\a}_{r_1+r_2+r_3+s_1+1}+1}y^{r_4}.
\end{align} 
$(d).$ When $r_3=r_4=0, r_2\neq 0$, we have new term replacing $r^{th}$ $y_1$ with $-tx$, i.e., $(r_1+r_2+1)^{th}$ $y$ by $-tx$ with the same coefficient as shuffle product. Therefore, we have 
$$\underbrace{y_1 \dots y_1}_{r_1}y_2(\underbrace{y_1 \dots y_1}_{r_2-1}(-tx) \SH \underbrace{y_2 \dots y_2}_{s_1-2}){y_2}^{s_2+1}$$
So, the new term is
\begin{align}
	&-t\sum_{\substack{r_1+r_2=r;\\r_1, r_2 \geq 1; \\{\a}_1+ \dots +{\a}_{r_1+r_2+s_1+s_2}=a+b_1+b_2\\ {\a}_1, \dots , {\a}_{r_1+r_2+s_1+s_2} \geq 0}} \Biggl\{\binom{{\a}_1}{a}\binom{r_2+s_1-2}{r_2}{\delta}_{{\a}_1+ \dots + \a_{r_1+1}, a+b_1}\prod_{k=r_1+2}^{r_1+r_2+s_1}{\delta}_{\a_k,0}\prod_{k=r_1+r_2+s_1+2}^{r+s_1+s_2}{\delta}_{\a_k,0}  \nonumber\\
	& \ \ \hspace{2cm}\times  \sum_{i=r_2-1}^{r_2+s_1-4} \binom{i}{r_2-1}\Biggr\}\times x^{\a_1}y \dots x^{\a_{r_1}+1}y^{i+1}xy^{r_2+s_1-i-2}x^{{\a}_{r_1+r_2+s_1+1}}y \dots x^{\a_{r+s_1+s_2}}y.
\end{align}
$(e).$ When $r_2=r_3=r_4=0$, we have new terms replacing $r^{th}$ $y_1$ by $-tx$ with the same coefficient as shuffle product
\begin{align}
	&-t\sum_{\substack{{\a}_1+ \dots +{\a}_{r+s_1+s_2}=a+b_1+b_2\\ {\a}_1, \dots , {\a}_{r+s_1+s_2} \geq 0}} \Biggl\{\binom{{\a}_1}{a}{\delta}_{{\a}_1+ \dots + \a_{r+1}, a+b_1}\prod_{k=r+2}^{r+s_1}{\delta}_{\a_k,0} \prod_{k=r+s_1+2}^{r+s_1+s_2}{\delta}_{\a_k,0} \Biggr\} \nonumber\\
	& \ \ \hspace{2cm} \times x^{\a_1}y \dots x^{{\a}_{r-1}}y   x^{\a_r+{\a}_{r+1}+1}yx^{\a_{r+2}}y \dots x^{\a_{r+s_1+s_2}}y.
\end{align} 
For case $(ii)$, we have the following subcases:\\
$(a).$ For $r_3 \neq 0$ and $s_2 \neq 1$, new term involving $t$ is,
\begin{align}
	&-t(1-\delta_{s_2,1})\sum_{\substack{r_1+r_2+r_3=r;\\r_1, r_3 \geq 1, r_2 \geq 0; \\{\a}_1+ \dots +{\a}_{r_1+r_2+s_1+1}=a+b_1+b_2\\ {\a}_1, \dots , {\a}_{r_1+r_2+r_3+s_1+1} \geq 0\\1 \leq l \leq s_1-1}} \Biggl\{\binom{{\a}_1}{b_1}\binom{r_1+s_1-l-2}{r_1-1}\sum_{i=\text{min}\{r_3, s_2-1\}-1}^{r_3+s_2-3}\bigg\{\binom{i}{r_3-1}+\binom{i}{s_2-2}\bigg\} \nonumber\\
	& \ \ \hspace{2cm} {\delta}_{{\a}_1+ \dots + \a_{l+1}, a+b_1}\prod_{k=l+2}^{r_1+s_1}{\delta}_{\a_k,0} \Biggr\}x^{\a_1}y \dots x^{{\a}_{r_1+r_2+s_1+1}}y^{i+1}xy^{r_3+s_2-i-2}.
\end{align}
$(b).$ When $r_3=0, r_2 \neq 0$, we have new terms replacing $r^{th}$ $y_1$ with $-tx$, i.e., $(r_1+r_2+s_1)^{th}$ $y$ by $-tx$ with the same coefficient as shuffle product, and given by
\begin{align}
	&-t\sum_{\substack{r_1+r_2=r;\\r_1, r_2 \geq 1; \\{\a}_1+ \dots +{\a}_{r_1+r_2+s_1+1}=a+b_1+b_2\\ {\a}_1, \dots , {\a}_{r_1+r_2+s_1+1} \geq 0\\1 \leq l \leq s_1-1}} \Biggl\{\binom{{\a}_1}{b_1}\binom{r_1+s_1-l-2}{r_1-1}{\delta}_{{\a}_1+ \dots + \a_{l+1}, a+b_1}\prod_{k=l+2}^{r_1+s_1}{\delta}_{\a_k,0} \Biggr\} \nonumber\\
	& \ \ \hspace{2cm} \times x^{\a_1}y \dots x^{{\a}_{r_1+r_2+s_1-1}}yx^{{\a}_{r_1+r_2+s_1}+{\a}_{r_1+r_2+s_1+1}+1}y^{s_2}.	
\end{align} 
$(c).$ When $s_2=1, r_3\neq 0$, we have new term replacing $(s_1+1)^{th}$ $y_2$ with $-tx$, i.e., $(r_1+r_2+s_1+1)^{th}$ $y$ by $-tx$ with the same coefficient as shuffle product, and are given by
\begin{align}
	&-t\delta_{s_2,1}\sum_{\substack{r_1+r_2+r_3=r;\\r_1, r_3 \geq 1, r_2 \geq 0; \\{\a}_1+ \dots +{\a}_{r_1+r_2+s_1+1}=a+b_1+b_2\\ {\a}_1, \dots , {\a}_{r_1+r_2+s_1+1} \geq 0\\1 \leq l \leq s_1-1}} \Biggl\{\binom{{\a}_1}{b_1}\binom{r_1+s_1-l-2}{r_1-1}{\delta}_{{\a}_1+ \dots + \a_{l+1}, a+b_1}\prod_{k=l+2}^{r_1+s_1}{\delta}_{\a_k,0} \Biggr\} \nonumber\\
	& \ \ \hspace{2cm} \times x^{\a_1}y \dots x^{{\a}_{r_1+r_2+s_1}}yx^{{\a}_{r_1+r_2+s_1+1}+1}y^{r_3}. 
\end{align} 
$(d).$ When $r_3=r_2=0, r_1\neq 1$, we have new term replacing $r^{th}$ $y_1$ with $-tx$, i.e., $(r_1+l+1)^{th}$ $y$ by $-tx$ with the same coefficient as shuffle product. Therefore, we have 
$$\underbrace{y_2 \dots y_2}_{l}y_1(\underbrace{y_1 \dots y_1}_{r_1-1}(-tx) \SH \underbrace{y_2 \dots y_2}_{s_1-l-1}){y_2}^{s_2+1}$$
So, the new terms are given by
\begin{align}
	&-t(1-\delta_{r,1})\sum_{\substack{{\a}_1+ \dots +{\a}_{r+s_1+s_2}=a+b_1+b_2\\ {\a}_1, \dots , {\a}_{r+s_1+s_2} \geq 0;\\1 \leq l \leq s_1-1}} 
	\Biggl\{\binom{{\a}_1}{b_1}{\delta}_{{\a}_1+ \dots + \a_{l+1}, a+b_1}\prod_{k=l+2}^{r_1+s_1}{\delta}_{\a_k,0}\prod_{k=r+s_1+2}^{r+s_1+s_2}{\delta}_{\a_k,0} \sum_{i=r-2}^{r+s_1-l-3}\binom{i}{r-2}\Biggr\} \nonumber\\
	& \ \ \hspace{2cm} \times x^{\a_1}y \dots x^{\a_{l+1}}y^{i+1}xy^{r+s_1-l-2} x^{{\a}_{r+s_1+1}}y \dots x^{{\a}_{r+s_1+s_2}}y.
\end{align}
$(e).$ When $r_2=r_3=0, r_1=1$, we have new term replacing $r^{th}$ $y_1$ with $-tx$, i.e., $(l+1)^{th}$ $y$ by $-tx$ with the same coefficient as shuffle product, and are given by
\begin{align}
	&-t\delta_{r,1}\sum_{\substack{{\a}_1+ \dots +{\a}_{r+s_1+1}=a+b_1+b_2\\ {\a}_1, \dots , {\a}_{r+s_1+1} \geq 0;\\1 \leq l \leq s_1-1}} 
	\Biggl\{\binom{{\a}_1}{b_1}{\delta}_{{\a}_1+ \dots + \a_{l+1}, a+b_1}\prod_{k=l+2}^{r+s_1}{\delta}_{\a_k,0}\Biggr\} \nonumber\\
	& \ \ \hspace{2cm} \times x^{\a_1}y \dots . x^{\a_l}yx^{\a_{l+1}+\a_{l+2}+1}yx^{\a_{l+3}}y \dots x^{\a_{r+s_1+1}}y^{s_2}.
\end{align}
For case $(iii)$, consider the following subcases:\\
$(a).$ For $r_2 \neq 0$ and $s_2 \neq 1$, new term involving $t$ is,
\begin{align}
	&-t(1-\delta_{s_2,1})\sum_{\substack{r_1+r_2=r;\\r_1, r_2 \geq 1; \\{\a}_1+ \dots +{\a}_{r_1+s_1+1}=a+b_1+b_2\\ {\a}_1, \dots , {\a}_{r_1+s_1+1} \geq 0}} \Biggl\{\binom{{\a}_1}{b_1}\binom{\a_{s_1+1}}{a+b_1-\sum_{i=1}^{s_1}\a_i}\sum_{i=\text{min}\{r_2, s_2-1\}-1}^{r_3+s_2-3}\bigg\{\binom{i}{r_2-1} \nonumber\\
	& \ \ \hspace{2cm} +\binom{i}{s_2-2}\bigg\}\Biggr\}  x^{\a_1}y \dots x^{{\a}_{r_1+s_1+1}}y^{i+1}xy^{r_2+s_2-i-2}.
\end{align}
$(b).$ When $r_2=0, r_1 \neq 0$, we have new terms replacing $r^{th}$ $y_1$ with $-tx$, i.e.,$(r_1+s_1)^{th}$  $y$ by $-tx$ with the same coefficient as shuffle product, and are given by
\begin{align}
	&-t\sum_{\substack{{\a}_1+ \dots +{\a}_{r+s_1+1}=a+b_1+b_2\\ {\a}_1, \dots , {\a}_{r+s_1+1} \geq 0}} \binom{{\a}_1}{b_1}\binom{\a_{s_1+1}}{a+b_1-\sum_{i=1}^{s_1}\a_i}  x^{\a_1}y \dots x^{{\a}_{r+s_1-1}}yx^{{\a}_{r+s_1}+{\a}_{r+s_1+1}+1}y^{s_2}.	
\end{align} 
$(c).$ When $s_2=1, r_2\neq 0$, we have new term replacing $(s_1+1)^{th}$ $y_2$ with $-tx$, i.e., $(r_1+s_1+1)^{th}$ $y$ by $-tx$ with the same coefficient as shuffle product, and these terms are given by
\begin{align}
	&-t\delta_{s_2,1}\sum_{\substack{r_1+r_2=r;\\r_1, r_2 \geq 1;\\{\a}_1+ \dots +{\a}_{r_1+s_1+1}=a+b_1+b_2\\ {\a}_1, \dots , {\a}_{r_1+s_1+1} \geq 0}} \binom{{\a}_1}{b_1}\binom{\a_{s_1+1}}{a+b_1-\sum_{i=1}^{s_1}\a_i}  x^{\a_1}y \dots x^{{\a}_{r_1+s_1}}yx^{{\a}_{r_1+s_1+1}+1}y^{r_2}.	 
\end{align} 
For case $(iv)$, we have the following subcases:\\
$(a).$ For $r \neq 1$ and $s_2 \neq l$, new term involving $t$ is,
\begin{align}
	&-t(1-\delta_{r,1})\sum_{\substack{{\a}_1+ \dots +{\a}_{s_1+l+1}=a+b_1+b_2\\ {\a}_1, \dots , {\a}_{s_1+l+1} \geq 0;\\1 \leq l < s_2}} \Biggl\{\binom{{\a}_1}{b_1}\binom{\a_{s_1+1}}{b_2}\sum_{i=\text{min}\{r-1, s_2-l\}-1}^{r+s_2-l-3}\bigg\{\binom{i}{r-2}+\binom{i}{s_2-l-1}\bigg\}\Biggr\} \nonumber\\
	& \ \ \hspace{2cm} \times  x^{\a_1}y \dots x^{{\a}_{s_1+l+1}}y^{i+1}xy^{r+s_2-l-i-2}.
\end{align}
$(b).$ When $r=1, s_2\neq l$, we have new term replacing $1^{st}$ $y_1$ with $-tx$, i.e., $(s_1+l+1)^{th}$  $y$ by $-tx$ with the same coefficient as shuffle product. Therefore, new terms are given by 
\begin{align}
	&-\delta_{r,1}t\sum_{\substack{{\a}_1+ \dots +{\a}_{s_1+l+1}=a+b_1+b_2\\ {\a}_1, \dots , {\a}_{s_1+l+1} \geq 0;\\1 \leq l < s_2}} \Biggl\{\binom{{\a}_1}{b_1}\binom{\a_{s_1+1}}{b_2}  \Biggr\}x^{\a_1}y \dots x^{{\a}_{s_1+l}}yx^{{\a}_{s_1+l+1}+1}y^{s_2-l}.	
\end{align} 
$(c).$ When $s_2=l$, we have new term replacing $(s_1+l)^{th}$ $y_2$ by $-tx$ with the same coefficient as shuffle product. Therefore, the new terms are given by
\begin{align}
	\label{p123}
	&-t\sum_{\substack{{\a}_1+ \dots +{\a}_{s_1+s_2+1}=a+b_1+b_2\\ {\a}_1, \dots , {\a}_{s_1+s_2+1} \geq 0}} \Biggl\{\binom{{\a}_1}{b_1}\binom{\a_{s_1+1}}{b_2}  \Biggr\}x^{\a_1}y \dots x^{{\a}_{s_1+s_2-1}}yx^{{\a}_{s_1+s_2}+{\a}_{s_1+s_2+1}+1}y^{r}.			 
\end{align} 
Therefore $\mathcal{P}$ is the sum of all expressions \eqref{p1234} to \eqref{p123}.\\

\textbf{Acknowledgment:} The research of first author is supported by the University Grants Commission (UGC), India through NET-JRF (Ref. No. 191620198830).
	
	{}

\end{document}